\documentclass{amsart}

\usepackage{amsmath}
\usepackage{amsthm}
\usepackage{amssymb}
\usepackage{enumitem}
\usepackage{mathrsfs}
\usepackage{mathrsfs}
\usepackage[mathscr]{euscript}

\newtheorem{theorem}{Theorem}[section]
\newtheorem*{theorem*}{Theorem}

\theoremstyle{plain}
\newtheorem{lemma}[theorem]{Lemma}

\theoremstyle{definition}
\newtheorem{definition}[theorem]{Definition}
\newtheorem{proposition}[theorem]{Proposition}

\newtheorem{observation}[theorem]{Observation}
\newtheorem*{observation*}{Observation}
\newtheorem{fact}[theorem]{Fact}
\newtheorem*{fact*}{Fact}

\newtheorem{remark}[theorem]{Remark}
\newtheorem*{remark*}{Remark}

\newtheorem{question}{Question}

\newcommand{\ef}{\e}

\newcommand{\base}{basis}
\newcommand{\bases}{bases}
\renewcommand{\bot}{{\base\ of order $2$}}

\newcommand{\bra}[1]{ \left( #1 \right) }
\newcommand{\abs}[1]{\left|#1\right|}
\newcommand{\fp}[1]{\left\{ #1 \right\}}
\newcommand{\norm}[1]{\left\lVert #1 \right\rVert}
\newcommand{\normLip}[1]{\left\lVert #1 \right\rVert_{\mathrm{Lip}}}
\newcommand{\fpa}[1]{\left\lVert #1 \right\rVert_{\mathbb{R}/\mathbb{Z}}}
\newcommand{\set}[2]{\left\{ #1 \ \middle| \ #2 \right\} }

\newcommand{\tendsto}[1]{\xrightarrow[#1]{}}

\renewcommand{\k}{\kappa}
\newcommand{\e}{\varepsilon}
\renewcommand{\a}{\alpha}
\renewcommand{\b}{\beta}
\newcommand{\de}{\delta}
	
\newcommand{\cC}{\mathscr{C}}
\newcommand{\cN}{\mathcal{N}}
\newcommand{\cA}{\mathcal{A}}
\newcommand{\cB}{\mathcal{B}}

\newcommand{\Borel}{\mathscr{B}}

\newcommand{\NN}{\mathbb{N}}
\newcommand{\QQ}{\mathbb{Q}}
\newcommand{\PP}{\mathbb{P}}
\newcommand{\EE}{\mathbb{E}}
\newcommand{\ZZ}{\mathbb{Z}}
\newcommand{\RR}{\mathbb{R}}
\newcommand{\TT}{\mathbb{T}}

\newcommand{\densNat}{\mathrm{d}}

\newcommand{\Mod}[1]{\bmod{#1}}

\newcommand{\A}[2]{\cA_{#1}^{#2} }

\makeatletter
\@namedef{subjclassname@2010}{
  \textup{2010} Mathematics Subject Classification}
\makeatother

\numberwithin{equation}{section}

\begin{document}

\begin{abstract}

	We study sets of the form $\cA = \big\{ n \in \NN \big| \fpa{p(n)} \leq \e(n) \big\}$ for various real valued polynomials $p$ and decay rates $\e$. In particular, we ask when such sets are \bases\ of finite order for the positive integers.
	
	We show that generically, $\cA$ is a \bot\ when $\deg p \geq 3$, but not when $\deg p = 2$, although then $\cA + \cA$ still has asymptotic density $1$.  
\end{abstract}

\title[Sets of recurrence as bases]{Sets of recurrence as bases for the positive integers}

\author[J. Konieczny]{Jakub Konieczny}
\address{Mathematical Institute \\ 
University of Oxford\\
Andrew Wiles Building \\
Radcliffe Observatory Quarter\\
Woodstock Road\\
Oxford\\
OX2 6GG}
\email{jakub.konieczny@gmail.com}

\date{}



\maketitle

\setcounter{section}{0}

\section*{Introduction}

Let $p(n)$ be a real polynomial and let $\e(n) > 0$ be a slowly decaying function. We consider the sets
$$
	\cA = \set{n \in \NN}{ \fpa{p(n)} \leq \e(n) },
$$
where $\fpa{t} = \min_{n \in \ZZ} \abs{t - n}$ denotes the distance to the nearest integer and $\NN = \{0,1,2\dots\}$. 

Our particular concern will be with the additive properties of such sets. Specifically, when is $\cA$  an \base\ for $\NN$ of a given finite order? That is, for which $k$, if any, is it true that the sumset
$$ k {} \cA = \cA + \cdots + \cA = \set{ n_1 + \cdots + n_k}{n_i \in \cA}$$
contains all sufficiently large integers? We will also be interested in when $\cA$ is an \emph{almost} \base\ of order $k$, by which we mean that $k {} \cA$ has asymptotic density $\densNat(k\cA)$ equal to $1$. Here, asymptotic density of a set $\cB$ is defined as 
$$
	\densNat(\cB) = \lim_{n\to \infty} \frac{\abs{ \cB \cap [n] }}{n},
$$
provided that the limit exists. We use the the symbol $[n]$ to denote the set $\{1,2,\dots,n\}$.

We consider two types of behaviour of $\e(n)$: we either demand that $\e(n) \to 0$, or that $\e(n)$ is bounded pointwise by a suitably small constant $\e_0$ (in which case we may equally well assume that $\e(n) = \e_0$). This technical issue will appear at various points in the paper. 

In the case when $\deg p = 1$, the problem is rather straightforward. We are essentially dealing with $\mathrm{Bohr}$ sets, which are simple and well studied objects (see e.g. \cite[Chapter 4.4]{Tao2006}). We expect that the sets $k {} \cA$ should not be significantly larger than $\cA$, and hence that $\cA$ should not be a \base\ of any order for sufficiently small $\e$.

It is an easy exercise to show that for any $k$ the set 
	$$ 
		\cA =  \set{n \in \NN}{ \fpa{ \a n } \leq \e(n) }
	$$
is not a \base\ of order $k$ provided that, say, $\alpha \in \mathbb{R} \setminus \mathbb{Q}$ and $\e(n) < \frac{1}{3k}$ for all $n$. Indeed, it follows easily from the observation that $\fpa{ N \a } < \frac{1}{3}$ for $N \in k \cA$. Similarly, one can show that $\cA$ defined above is not a \base\ of order $k$ if $\e(n) \to 0$ as $n \to \infty$. We leave the details to the interested reader.
 
The problem is most interesting when $\deg p = 2$. One might expect $\cA$ to behave roughly as a random set such that $n \in \cA$ with probability $\e(n)$, and hence to be a \base\ of finite order if $\e(n)$ decays reasonably slowly. A particular case of this problem was considered by Erd\H{o}s, who asked the following\footnote{Personal communication from Ben Green; no written reference could be located.}.

\begin{question}
	Is the set $\cA =  \set{n \in \NN}{ \fpa{ \sqrt{2} n^2 } \leq \frac{1}{\log n }}$ a \base\ of order $2$?
\end{question}

Somewhat unexpectedly, the answer to this question is negative. One can even produce an explicit sequence $N_i = \frac{(3+2\sqrt{2})^{2i+1} - (3-2\sqrt{2})^{2i+1}}{2 \sqrt{2} }$ such that $N_i \not \in 2 \cA$ for all sufficiently large $i$. 

Several other constructions of this type are possible, each leading to a sequence $N_i \not \in 2 \cA$ with $N_i$ growing exponentially with $i$.  Hence, one may hope that the following weaker variant should have a positive answer. Recall that we call $\cA$ an almost \base\ of order $2$ if $\densNat(2\cA) = 1$.

\begin{question}
	Is the set $\cA =  \set{n \in \NN}{ \fpa{ \sqrt{2} n^2 } \leq \frac{1}{\log n }}$ an almost  \base\ of order $2$?
\end{question}

This is indeed the case. In fact, we can prove a stronger statement concerning the size of the complement $(2\cA)^c = \NN \setminus 2\cA$, namely that as $T \to \infty$ we have $\abs{ [T] \setminus 2\cA} \ll \log^C T$, where $C$ is a constant.

Here and elsewhere, we use the Vinogradov notation $f \ll g$, as well as the more standard $f = O(g)$, to denote the statement that $f \leq C g$, for some constant $C$. When $C$ depends on a parameter $A$, we write $f \ll_{A} g$ or $f = O_A(g)$. If $f = O(g)$ and $g = O(f)$, we write $f = \Theta(g)$. 

In larger generality, we have the following collection of results.

\newtheorem*{theoremA}{Theorem A}
\begin{theoremA}\label{intro::thm:A}
	Let $\e(n)$ be a slowly-decaying function, let $\a \in \RR \setminus \QQ$ and set:
	$$ 
		\cA =  \set{n \in \NN}{ \fpa{ \a n^2 } \leq \e(n) }.
	$$
	Then the following are true:
	\begin{enumerate}[label=\textbf{A\arabic*.}, ref=\textbf{A\arabic*}]
	\item\label{intro::thm:A:1} 
	For any $\a \in \RR \setminus \QQ$, $\cA$ is an almost \base\ of order $2$, provided that $\e(n)$ decays slowly enough. 

	\item\label{intro::thm:A:3} Moreover, for uncountably many exceptional values of $\a$, $\cA$ is a \base\ of order $2$, provided that $\e(n)$ decays slowly enough. 

	\item\label{intro::thm:A:4} In particular, for any $\a \in \RR \setminus \QQ$, $\cA$ is a \base\ of order $3$, provided that $\e(n)$ decays slowly enough.

	\item\label{intro::thm:A:2} However, for almost all $\a$, $\cA$ is not a \base\ of order $2$, as long as $\e(n) \to 0$. 
	
	\end{enumerate}
			
\end{theoremA}

Above, the phrase ``provided that $\e(n)$ decays slowly enough'' may be expanded into ``there exists $\e_0(n) \to 0$ such that if $\e(n) \geq \e_0(n)$ for all $n$, then the statement holds'', and ``almost all'' means ``all except for a set of Lebesgue measure $0$''. We state the results in a more rigorous manner when we approach the proof.

We first address item \ref{intro::thm:A:2}, which was the original motivation for this research project. Because of \ref{intro::thm:A:3}, we cannot hope to obtain a result for all $\a$, but we are able to cover a number of interesting cases, including Lebesgue-almost all reals, as well as all quadratic surds. This is done in Section \ref{neg::section}. 

Items \ref{intro::thm:A:1} and \ref{intro::thm:A:3} are proved in Section \ref{pos::section}. Our key idea is to translate information about the complement of $2 {} \cA$ into information about good rational appoximations of $\a$. We are then able to use known equidistribution results as a black box, in order to show that if $(2 {} \cA)^c$ had positive (upper asymptotic) density, then $\a$ would have too many good rational approximations. Item \ref{intro::thm:A:3} is proved by an explicit construction using continued fractions. 

Item \ref{intro::thm:A:4} is an immediate consequence of \ref{intro::thm:A:1} (or, strictly speaking, the proof thereof). In fact, our argument implies that $2 \cA + \cB$ contains all sufficiently large integers for any set $\cB$ with at least $2$ elements.

After this paper was completed, the author learnt that in \cite{Deshouillers-1}, Deshouillers, Erd\H{o}s and S\'{a}rk\"{o}zy show that \ref{intro::thm:A:4} holds for $\alpha = \frac{\sqrt{5}+1}{2}$ with $\e(n) \sim n^{-1/12}$ (which is much better than the convergence rate which could be extracted from the argument in this paper); for related results see also \cite{Deshouillers-2}. 

\mbox{}

For polynomials of higher degrees $\deg p \geq 3$, the situation becomes much simpler. The heuristic expectation that $\cA$ should be a \base\ of order $2$ is accurate in this case, as long as we impose the suitable genericity assumptions. 
 Below we give a special case of our main result for polynomials of degree $\geq 3$.

\newtheorem*{theoremB}{Theorem B}\label{intro::thm:B}
\begin{theoremB}
	Let $d \geq 3$. Fix some slowly-decaying function $\e(n)$, $\a \in \RR \setminus \QQ$ and set:
	$$ 
		\cA =  \set{n \in \NN}{ \fpa{ \a n^d } \leq \e(n) }.
	$$ 
	Then the following are true:
	\begin{enumerate}[label=\textbf{B\arabic*.}, ref=\textbf{B\arabic*}]
	\item\label{intro::thm:B:1} For almost all $\a$, $\cA$ is a \base\ of order $2$, provided that $\e(n)$ decays slowly enough.
	\item\label{intro::thm:B:2} Nevertheless, for uncountably many $\a$, $\cA$ is not a \base\ of order $2$, even when $\e(n)$ is constant.
	\end{enumerate}
\end{theoremB}

	In Section \ref{hd::section} we will establish a more general result in which the polynomial $p$ varies in a linear family. The bulk of the difficulty lies in proving \ref{intro::thm:B:1}. We rely on similar ideas as for \ref{intro::thm:A:1}, and relate each element in the complement of $2 \cA$ to the lack of equidistribution of a certain polynomial sequence. Using known result about distribution of polynomial sequences, we then connect lack of equidistribution with a system of approximate rational dependencies, which generically turn out not to be satisfiable.
	
	For \ref{intro::thm:B:2}, it suffices to take $\a$ sufficiently well approximable by rationals, and we can construct such $\a$ explicitly.

\subsection*{Acknowledgements} The author wishes to thank Ben Green for introducing him to the problem and for much helpful advice and corrections to the manuscript. The author is also indebted to Bryna Kra for comments on possible further directions and to Jean-Marc Deshouillers for bringing relevant references to his attention. Finally, thanks go to Sean Eberhard, Frederick Manners, Przemys\l{}aw Mazur and Rudi Mrazovi\'c for many informal discussions.


\section{Failure to be a \base\ of order $2$.}\label{neg::section}

Our goal in this section is to prove that the sets
\begin{equation}
	\A{\e}{\a} := \set{n \in \NN}{ \fpa{\a n^2 } < \e(n)}
	\label{neg::eq::def-of-A}
\end{equation}
are ``usually'' not \bases\ of order $2$, even when $\e(n) = \e_0$ is constant.

\begin{theorem*}[\ref{intro::thm:A:2}, reiterated]
There exists a set $Z \subset \RR$ of Lebesgue measure $0$ such that for any $\a \in \RR \setminus Z$ and for any $\e(n) \to 0$, the set $\A{\e}{\a}$ defined in \eqref{neg::eq::def-of-A} is not a \base\ of order $2$.

Moreover, the same statement is true for $\a \in \QQ[\sqrt{d}]\setminus \QQ$, for any $d \in \NN$.
\end{theorem*}

This result is somewhat surprising, because a random (unstructured) set of similar size should be a \bot. In fact, if $\cA \subset \NN$ is constructed randomly with $\PP(n \in \cA) = \e$, independently for each $n$, then with high probability $2 {} \cA$ contains all integers larger than roughly $\frac{1}{\e^2} \log \frac 1\e$. 

\mbox{}\\

We will prove a variety of partial results, with different restrictions on $\a$ and $\e(n)$, not all of which are included in Theorem \ref{intro::thm:A:2} as stated above. For $\a$, we separately address the ``structured'' case when $\a$ is a quadratic surd or, more generally, is badly approximable, and the ``generic'' case when $\a$ is selected from a suitable set of full measure. For $\e(n)$, we either assume that $\e(n) \to 0$ as $n \to \infty$ or that $\e(n) \leq \e_0(\a)$ is bounded by a constant which is allowed to depend on $\a$. 

\subsection{General strategy}

	We begin by introducing a somewhat technical tool which will allow us to detect large integers $N$ in the complement of $2\A{\e}{\a}$. Importantly, we are able to reduce the task of proving that $N \not \in 2\A{\e}{\a}$ to the task of verifying a simple Diophantine inequality.

	The basic idea is quite simple. 
	Suppose that we allowed $\a$ to take rational values, and take for instance $\a = \frac{1}{2}$. Assuming that $\e(n) < \frac{1}{2}$ for all $n$, the set $\A{\e}{\a}$ is far from being a \bot. Indeed, we then have $\A{\e}{\a} = 2 \NN$, which is not a \base\ of any order.
	
	The following lemma makes this observation quantitative. We will use it multiple times.

\begin{lemma}\label{neg::prop:obst-eps-adv}
	Suppose that for an odd integer $N$, there are integers $k,m$, with $k$ even and $m$ odd, and a real parameter $\delta > 0$, such that we have 
	\begin{equation}
	 \fpa{N\alpha - \frac{m}{k} } < \frac{1- \delta}{k N}.
	\end{equation}
	Then $N \not \in 2 {} \A{\e}{\a}$ for any pointwise bounded $\e(n) \leq \e_0$, where $\e_0 = \frac{\delta}{2 k}$.
\end{lemma}
\begin{proof}

	Let us take $\gamma$ with $\abs{\gamma} < 1-\delta$ so that $N \alpha \equiv \frac{m}{k} + \frac{\gamma}{kN} \pmod{1}$. Consider any decomposition $N = n_1 + n_2$ with $n_1,n_2 \in \NN$. We can then compute:
	\begin{align*}
	 \fpa{n_1^2 \alpha- n_2^2\alpha} 
	 &= \fpa{(n_1 - n_2) N \alpha}
	\\& = \fpa{ \frac{(n_1-n_2) m}{k} + \frac{n_1-n_2}{N} \frac{\gamma}{k} } 
	 \geq \frac{1}{k}  - \frac{\abs{\gamma}}{k} > \frac{\delta}{k} = 2 \e_0.
	\end{align*}	
	
	It follows that $n_1^2 \alpha$ and $ n_2 ^2 \alpha \bmod{1}$ cannot both lie in $(-\e_0,\e_0) \bmod{1}$. Hence at least one of $n_1,n_2$ fails to belong to $\A{\e}{\a}$ and consequently $N \not \in 2 {} \A{\e}{\a}$. 
\end{proof}

Our next result is in similar spirit, with the difference that instead of pointwise bound $\e(n) \leq \e_0$, we work with the condition $\e(n) \to 0$.

\begin{remark*}\label{neg::remark::e-const-vs-e-var}
	It might seem that a set $\A{\e}{\a}$ with $\e(n) \to 0$ must necessarily be ``smaller'' than one with constant $\e(n) = \e_0$, and hence that Lemma \ref{neg::prop:obst-var} below is strictly weaker than Lemma \ref{neg::prop:obst-eps-adv}. However, we wish to emphasise that for variable $\e(n)$ we allow the value $\e(n)$ to be large when $n$ is small.
	
	Because we expect the complement of $2 {} \A{\e}{\a}$ to have density $0$, we cannot rule out a priori that small values of $n$ play a role. In fact, for any $\e_0 > 0$, one can construct $\e(n) \to \e_0$ such that $\A{\e}{\a}$ is \bot, simply by exploiting the fact that $ \A{\e_0}{\a}$ is syndetic. Hence, it is not the case that small values of $n$ can be altogether ignored.
	
	Here and elsewhere, by a slight abuse of notation, we write $\A{\e_0}{\a}$, allowing the symbol $\e_0$ to also denote the constant function $n \mapsto \e_0$. 
\end{remark*}

\begin{lemma}\label{neg::prop:obst-var}
	Let $\ef(n) \to 0$, and let $\left(N_i\right)_{i = 1}^{\infty}$ be an increasing sequence of odd integers. Suppose that for each $i$, we have 
	$$N_i \alpha = \frac{m_i}{ k} + \frac{\gamma_i}{k N_i},$$ 
	where $m_i,\ k$ are integers, $k$ is even and $m_i$ is odd. Assume further that $\gamma$ with $\abs{\gamma} < 1$ is an accumulation point of $\gamma_i$. Then $N_i \not \in 2 {} \A{\e}{\a}$ for infinitely many $i$, unless $\gamma + kn^2 \alpha \in \ZZ$ for some integer $n$.
\end{lemma}
\begin{proof}
	Passing to a subsequence, we may assume that $\gamma_i \to \gamma$ as $i \to \infty$.
	
	Let $\e_0$ be such that $\frac{1-\abs{\gamma}}{k} > 2\e_0 > 0$. Then from previous Proposition \ref{neg::prop:obst-eps-adv} it follows that $N_i \not \in 2 {}  \A{\e_0}{\a}$ for sufficiently large $i$. Hence, if $N_i \in 2 {} \A{\e}{\a}$ for some $i$, then $N_i$ needs to have a representation as $n_1 + n_2$ with $n_1 \in \A{\e}{\a} \setminus \A{\e_0}{\a},\ n_2 \in \A{\e}{\a}$. Note that the set $\A{\e}{\a} \setminus \A{\e_0}{\a}$ is finite, so passing to a subsequence again we may assume that there exists a single $n_1$ such that for each $i$ we have $n_{2,i} := N_i - n_1 \in \A{\e}{\a}$.
	
	Directly from the membership condition for $\A{\e}{\a}$, we now find
	$$ \e(n_{2,i}) > \fpa{ (N_i-2n_1) \left( \frac{m_i}{k} + \frac{\gamma_i}{kN_i} \right)+ n_1^2\a}. $$
	We have $\e(n_{2,i}) \to 0$  and $\frac{N_i-2n_1}{N_i} \frac{\gamma_i}{k} \to \frac{\gamma}{k}$ as $i \to \infty$. It follows that
		$$ \fpa{ \frac{m_i'}{k} + \frac{\gamma}{k} + n_1^2\a} \to 0, $$
where $m_i' := (N-2n_1)m_i \mod k$. Note that $m_i'$ is odd and takes only finitely many values. Restricting to a subsequence, we may assume that $m_i' = m'$ is constant. Now, the expression in the limit above is independent of $i$, and hence
	$$\fpa{ \frac{m'}{k} + \frac{\gamma}{k} + n_1^2\a} = 0.$$
In particular, we have
$$ k\left( \frac{\gamma}{k} + n_1^2\a \right) \in \ZZ,$$
contradicting the irrationality assumption.
\end{proof}

\subsection{Quadratic irrationals}
\label{neg::sq(d)::section}

We will now prove Theorem \ref{intro::thm:A:2} in the special case when $\a = \sqrt{2}$. The argument generalises to $\a \in \QQ[\sqrt{d}] \setminus \QQ$ without any new ideas. This case is already representative for some of our methods. Our immediate goal is the following result.

\begin{proposition} \label{neg::sqrt::prop:main}
	Let $\e_1 := \frac{1}{4}(1 - \frac{1}{4 \sqrt{2}})$. Suppose that either we have the pointwise bound $\e(n)  \leq \e_0 < \e_1$ or that $\e(n) \to 0$. Then $\A{\e}{\sqrt{2}}$ is not a \bot. In the case when $\e(n)$ is pointwise bounded, we additionally have the quantitative bound
	$$\abs{[T] \setminus 2 {} \A{\e}{\sqrt{2}} } \gg \log T,$$
where the implicit constant depends at most on $\e_0,\e_1$.
\end{proposition}

\begin{proof}

Any positive integer solution $(x,y)$ to the Pell equation
\begin{equation}
 X^2 - 2 Y^2 = 1, 
 \label{neg::sq(d)::eq:Pell}
\end{equation}
gives rise to the rational approximation $\frac{x}{y}$ of $\a$ with $\frac{x}{y} - \sqrt{2} = \frac{1}{y(x+\sqrt{2}y)}$. 

The fundamental solution to \eqref{neg::sq(d)::eq:Pell} is $(x,y)=(3,2)$. If we let $\phi := 3 + 2\sqrt{2} $ and $\hat\phi := 3 - 2\sqrt{2}$, and define integer sequences $a_i,$ $b_i$ by $\phi^i = a_i + b_i \sqrt{2}$, then all solutions to \eqref{neg::sq(d)::eq:Pell} are of the form $(x,y) = (a_i,b_i)$ We note that $a_i$, $b_i$ have explicit formulas:
\begin{equation}
a_i = \frac{\phi^{i} + \hat\phi^{i}}{2},\quad b_i = \frac{\phi^{i} - \hat\phi^{i}}{2 \sqrt{2} },
\end{equation}
as well as recursive relations:
	\begin{align*}
		a_{i+2} &= 6 a_{i+1} - a_i, &\quad a_0=1,\ a_1 = 3, \\
		b_{i+2} &= 6 b_{i+1} - b_i, &\quad b_0=0,\ b_1 = 2 .
	\end{align*}
	
It will be convenient to take $N_i = b_i/2$. For any $i$, $N_i$ is an integer, and if $i$ is odd, then $N_i$ is odd. We may write	
\begin{align}
	N_i \sqrt{2} = \frac{a_i}{2} + \frac{\gamma_i}{2N_i},  \qquad \text{where }
\gamma_i &=  \hat\phi^i N_i = \frac{-1}{4\sqrt{2}} + O\left(\frac{1}{N_i^2}\right).
\end{align}

To prove the statement in the case when $\e(n) \leq \e_0 < \e_1$ is pointwise bounded, we apply Lemma \ref{neg::prop:obst-eps-adv} to $N_i$, assuming that $i$ is large enough and odd. It follows that $N_i \not \in 2 {} \A{\e}{\sqrt{2}}$, and hence $\A{\e}{\sqrt{2}}$ is not a \bot. The quantitative estimate follows from the fact that $N_i = \Theta(\phi^i)$, and for any $T$ there are $\Theta(\log T)$ values of $i$ with $N_i < T$.
	
To prove the statement in the case when $\e(n) \to 0$, we similarly apply Lemma \ref{neg::prop:obst-var} to the sequence $N_i$ restricted to odd $i$, with $\gamma = \frac{-1}{4\sqrt{2}}$. The claim follows, unless there exists $n$ such that $\gamma + 2 n^2\sqrt{2} \in \ZZ$. Since $\sqrt{2}$ is irrational, that would imply that $2 n^2 = \frac{1}{8} $, which is absurd. 

\end{proof}

The result for general quadratic irrational $\a \in \QQ[\sqrt{d}]$ can be obtained with essentially the same argument. 

\begin{proposition} \label{neg::sqrt::prop:main-2}
	For any $\a \in \QQ[\sqrt{d}]\setminus \QQ$ there exists $\e_1 = \e_1(\a)$ such that the following is true. Suppose that either $\e(n) \leq \e_0 < \e_1$ or that $\e(n) \to 0$. Then $\A{\e}{\a}$ is not a \bot.
\end{proposition}
\begin{proof}
We may write $\alpha = \frac{a+b\sqrt{d}}{c}$, where $a,b,c$ are integers. Let $\phi = x + y \sqrt{d} \in \ZZ[\sqrt{d}]$ be a unit, and let $\mu := \nu_2(y)$, the largest power of $2$ dividing $y$. Replacing $\phi$ by $\phi^{2^n}$ for large $n$, we may assume that $\mu$ is sufficiently large, or more concretely that $\mu > \nu_2(b)$ and $2^\mu > bc$. 

Like before, we consider the integer valued sequences: 
\begin{equation}
a_i = \frac{\alpha \phi^i - \hat \alpha \hat \phi^i}{2 \sqrt{d}}c,\quad
b_i = \frac{\phi^i- \hat \phi^i}{2 \sqrt{d}}c.
\end{equation}

We have, using $\mu>\nu_2(b)$ and $\nu_2(x) = 0$, the relations:
\begin{equation}
	\nu_2(a_1) = \nu_2(ay + bx) = \nu_2(b),\quad \nu_2(b_1) = \nu_2(cy) = \nu_2(c) + \mu.
\end{equation}

Because the sequences $a_i$ and $b_i$ are periodic modulo any power of $2$, there exists some $L$ such that for all $i \equiv 1 \pmod{L}$ we have $\nu_2(a_i) = \nu_2(a_1)$ and $\nu_2(b_i) = \nu_2(b_1)$. 
For any such $i$ we define
\begin{equation}
	 N_i := \frac{b_i}{2^{\nu_2(c) + \mu}},\quad m_i := \frac{a_i}{2^{\nu_2(b)}},\quad k:= 2^{\nu_2(c) - \nu_2(b) + \mu}.
\end{equation}  	 
It is straightforward, if mundane, to check that these quantities are integers, and that we have the relation
	$$
		N_i\alpha = \frac{m_i}{k} + \frac{\gamma_i}{k N_i },
	$$
	where $\gcd(m_i,k) = 1$ and $\gamma_i$ are given by
	$$ \gamma_i = \frac{b\hat\phi^i}{2^{\mu+\nu_2(c)}} =
	\frac{\pm bc}{2^{\mu+\nu_2(bc)+1}\sqrt{d}} + o(1).
	$$
	Here, $o(1)$ denotes an error term which goes to $0$ as $i \to \infty$. The choice of $\mu$ guarantees that $\abs{\gamma_i} < \frac{1}{2}$ for large $i$. In the case $\e(n) \leq \e_0$, it follows from Lemma \ref{neg::prop:obst-eps-adv} that $N_i \not \in 2 {} \A{\e}{\a}$, provided that we take $\e_1 \leq \frac{1}{4 k}$.
	
	To deal with the case $\e(n) \to 0$, we notice that $\gamma_i \to \gamma := \frac{\pm bc \sqrt{d} }{2^{\mu+\nu_2(bc)+1} d}$. By Lemma \ref{neg::prop:obst-var}, we have $N_i \not \in 2 {} \A{\e}{\a}$ for sufficiently large $i$, unless $\gamma + k n^2 \a \in \ZZ$ for some $n$.  If it was the case that $\gamma + k n^2 \a \in \ZZ$ for some $n$, then it would follow that $ k2^{\mu+\nu_2(bc)+1} d \mid bc^2 $. However, this is impossible, since 
	$$\nu_2( k2^{\mu+\nu_2(bc)+1} d) \geq \mu + \nu_2(bc) + 1 > \nu_2(bc^2). \qedhere$$
	
\end{proof}

\subsection{Badly approximable reals}\label{neg::cf::section}

We now turn to the proof of a variant of Theorem \ref{intro::thm:A:2} for badly approximable values of $\a$. 

We say that $\a$ is \emph{badly approximable} if for any $p,q$ we have $$\abs{ \a - \frac{p}{q} } \geq \frac{c(\a)}{q^2},$$ where $c(\a) > 0$ is a constant dependent only on $\a$. The most well known example of such numbers are quadratic irrationals.

This is a more general situation than $\a \in \QQ[\sqrt{d}]$, but still rather specific. In particular almost all $\a$ are \emph{not} badly approximable. However, badly approximable $\a$ provide non-trivial and fairly explicit class of examples when the conclusion of Theorem \ref{intro::thm:A:2} holds (as opposed to an ``almost surely'' type of statement).

A useful characterisation  of badly approximable reals is that these are precisely the ones whose continued fraction expansion has bounded entries, see \ref{cf:bad-approx}. A specific class of badly approximable real numbers which has attracted some attention are those whose entries are produced by a finite automata. For instance, it has been shown that such numbers are transcendental, unless their continued fraction expansion is periodic, see \cite{Bugeaud2013}.

\mbox{}\\

The main result in this section shows that the sets $\A{\e}{\a}$ are not \bases\ of order $2$ for badly approximable $\a$ and sufficiently small $\e$. 

\begin{proposition}\label{neg::cf:prop:bad-approx=>not-bot}\label{neg::badapp::prop::main}
	If $\a$ is badly approximable then there is $\e_1=\e_1(\a)$ such that if $\e(n) \leq \e_0 < \e_1$
then $\A{\e}{\a}$ is not a \bot.

	Moreover, we have $\abs{[T] \setminus 2 {} \A{\e}{\a}} \gg \log T$, where the implicit constant depends only on $\a$.	
\end{proposition}

We will make extensive use of the continued fraction expansion of $2\a$. The crucial role played by the continued fraction expansion explains why we were able to give rather elementary proofs for $\a = \sqrt{2}$ and $\a \in \QQ[\sqrt{d}]$, whose expansion is particularly simple. 

Continued fractions are a classical topic, and we assume some familiarity with the basic notions and theorems. For an accessible introduction, see e.g. \cite{Khinchin2009}, or the more analytic approach in \cite{Khovanskii}. For the perspective inspired by measurable dynamics, see \cite[Chpt. 3]{Einsiedler2010}. 
We delegate a complete list of used properties to the Appendix \ref{cf::section}. Here, we just review several basic properties and introduce notation, which we will also use in subsequent sections. We will write:
\begin{align}
	2\a = [a_0;a_1,a_2,\dots] = a_0+\cfrac{1}{a_1+\cfrac{1}{a_2+\cdots}}
\end{align}
and $a_i$ will denote the coefficients of $2\a$ throughout this section (note that for technical reasons we consider $2\a$ rather than $\a$). Using obvious translation invariance, we may assume without loss of generality that $a_0 = 0$. We also denote the partial approximations:
\begin{align}
\frac{p_i}{q_i} = [0;a_1,a_2, \dots, a_i] = \cfrac{1}{a_1+\cfrac{1}{a_2+\cfrac{\dots}{a_i}}}.
\end{align}
These are essentially the best possible rational approximations of $\a$ (see \ref{cf:best-approx}), and we have the error term of the form
\begin{equation}
2\a = \frac{p_i}{q_i} + \frac{\delta_i}{q_i^2},
\end{equation}
where $\delta_i$ can be explicitly described by:
\begin{equation}
\abs{\delta_i} = q_i^2 \left( \frac{1}{q_i q_{i+1}} - \frac{1}{q_{i+1} q_{i+2}}+
\frac{1}{q_{i+2} q_{i+3}} - \dots \right).
\label{neg::eq::def-of-delta}
\end{equation} 
In particular, we have $\abs{\delta_i} < \frac{q_i}{q_{i+1}} < \frac{1}{a_{i+1}} \leq 1$.

\begin{proof}[Proof of Proposition \ref{neg::cf:prop:bad-approx=>not-bot}
.]
\mbox{}

	Because $\a$ is badly approximable, so is $2\a$, and by \ref{cf::quadratic-irrational} the coefficients $a_i$ are bounded, say $a_i \leq a_{\max}$. Let $\k$ be large enough that $2^\k \nmid a_i$ for all $i$.
	
	We claim that among any $4$ consecutive indices $\{j,j+1,j+2,j+3\}$ we can find an index $i$ such that $p_i$ is odd and $\nu_2(q_i) < \k$.
	
	For each $i$ we have (e.g. by \ref{cf:Delta}) that $\gcd(p_i,p_{i+1}) = \gcd(q_i,q_{i+1}) = 1$, and in particular neither $p_i,p_{i+1}$ nor $q_i,q_{i+1}$ can both be even. If for some $j \leq i \leq j+2$ we have that both $p_i$ and $p_{i+1}$ are odd then either $q_i$ or $q_{i+1}$ is odd, and the claim holds. Otherwise, since $p_i,p_{i+1}$ can never both be even, the parity of $p_i$ alternates. It follows that for one of $i \in \{j,j+1\}$, both $p_i,p_{i+2}$ are odd, while $p_{i+1}$ is even. Then $q_{i+2} - q_{i} = a_{i+2} q_{i+1}$ is not divisible by $2^\kappa$, so one of $q_{i+2},q_{i}$ is not divisible by $2^\k$. Either $i$ or $i+2$ is the sought index.
	
	Suppose now that $i$ is an index such that $p_i$ is odd and  $\nu_2(q_i) < \k$, whose existence we have just proved. Let us take
\begin{equation}
\label{neg::eq:def-of-N,k}
 N_i := \frac{q_i}{k_i/2},\qquad k_i := 2^{\nu_2(q_i) + 1},
\end{equation}
where as usual $\nu_2(q_i)$ is the largest power of $2$ dividing $q_i$. (The definition makes sense for arbitrary $i$, but we only apply it to $i$ as above.)

 Note that $k_i$ is guaranteed to be an even integer, and $N_i$ --- an odd integer. More precisely, $k_i \leq 2^\kappa$ is a power of $2$. It is slightly inconvenient that $N_i$ do not need to be distinct for distinct $i$, but this will not lead to problems since $N_i$ take any given value at most $\kappa$ times.

	Finally, we introduce $\gamma_i := 2 \delta_i/k_i$, so that we have the relation 
\begin{equation*}
\label{neg::eq:def-of-gamma}
N_i \alpha = \frac{p_i}{{k_i}} + \frac{\gamma_i}{{k_i} N_i}.
\end{equation*}

Note that we have the bounds
	$$
		\abs{\gamma_i} \leq \abs{\delta_i} \leq \frac{q_{i}}{q_{i+1}} = \frac{q_i}{q_i + q_{i-1}} \leq 1 - \frac{1}{1 + a_{\max} } < 1.
	$$

	We are now in position to apply Lemma \ref{neg::prop:obst-eps-adv} (with $\delta = \frac{1}{1 + a_{\max} }$). If follows that for $\e_0 < \frac{1}{2^{\kappa+1} (1 + a_{\max}) }$, if $\e(n) \leq \e_0$ for all $n$, then $N_i \not \in 2 {} \A{\e}{\a}$ for all $i$ as described above. In particular, the complement of $2 {} \A{\e}{\a}$ is infinite, proving the first part of the proposition.
	
	For the quantitative bound, we begin by noticing that $\log N_i = \Theta(i)$. Hence, given $T$, we have $N_i \in [T]$ for $i \leq i_0(T)$, with $i_0(T) = \Theta(\log T)$. For any $4$ consecutive values of $i$, sufficiently large, for at least one of them we have $N_i \not \in 2 {} \A{\e}{\a}$. 
	 Thus, 
	$$\abs{[T] \setminus 2 {} \A{\e}{\a}} \gg \frac{i_0(T)}{4 \k} \gg \log T. \qedhere$$
\end{proof}

\begin{remark*}
	In the above result we deal exclusively with pointwise bounded $\e(n)$. As noted earlier, it does not quite follow that analogous claim holds when $\e(n) \to 0$, since large values of $\e(n)$ for small $n$ can lead to problems. The main difficulty which stops us from extending our results is establishing the irrationality condition in Lemma \ref{neg::prop:obst-var}. This can be done for specific values of $\a$, but we do not give a general result.
\end{remark*}

\begin{remark*}
We believe that our methods should extend to numbers such as 
\begin{align*}
e^{\frac{1}{n}} &= [1; n - 1, 1, 1, 3 n- 1, 1, 1, 5n-1, 1,... ],\\
  \tanh (1/n) &= [0; n, 3n, 5n,7n,\dots],
\end{align*}
whose continued fraction expansions are well understood (see e.g. \cite[Chapter II]{Khovanskii}). It is a straightforward to adopt our argument to these situations, and the only reason why we do not pursue this further is that we doubt if any particular one of those results would be of much interest.
\end{remark*}

\subsection{Generic reals}

	Finally, we consider ``generic'' values of $\a$. We prove a version of \ref{intro::thm:A:2} which is valid for $\a$ outside of a set of measure $0$. Conveniently, in this case we can make the dependence on $\e$ rather explicit.
	
\begin{proposition}\label{neg::cf::cor:rand->not-bot-var}
\label{neg::cf::cor:rand->not-bot}\label{neg::gen::prop::main}
	For all $\a \in \RR$ except for a set of measure $0$, the set $\A{\e}{\a}$ fails to be a \bot\ if either $\e(n) \leq \e_0 < \frac{1}{4}$ for all $n$, or if $\e(n) \to 0$. 
\end{proposition}

We retain definitions and conventions from the previous section. Namely, we assume that $2\a \in (0,1)$ has expansion $2\a = [0;a_1,a_2,\dots]$ and $\frac{p_i}{q_i} = [0;a_1,a_2,\dots,a_i]$ and are the convergents.

The following description of the continued fraction expansion comes as no surprise. It can be construed as a continued fractions analogue of the fact that almost all numbers are normal. 
\begin{proposition}\label{neg::prop:normality}
	There exists a set $Z$ of zero measure such that the following is true for $\a \not \in Z$.

	Let $b = (b_i)_{i=1}^l \in \NN^l$ be a finite string of integers. Let $J$ be the set of indices where $b$ occurs in the expansion $(a_i)_{i=1}^{\infty}$, i.e. the set of those $j \in \NN$ for which $a_{j+t} = b_t$ for all $t \in [l]$. Then the asymptotic density $\densNat(J) = \lim_{n \to \infty} \frac{1}{n} \abs{J \cap [n]}$ of the set $J$ exists and is positive.
\end{proposition}
\begin{proof}
	Let $T\colon [0,1] \to [0,1]$ be the continued fraction map $T(x) = \fp{ \frac{1}{x} }$, where $\{\cdot \}$ denotes the fractional part, and let $\mu$ be the Gauss measure on $[0,1]$, $\mu(E) = \frac{1}{\log 2} \int_{E} \frac{dx}{x+1}$. 
	
	It is known that $([0,1],T,\Borel,\mu)$ is an ergodic measure preserving system, and that $T$ acts on continued fraction expansions by a shift: $T([0;c_1,c_2,\dots]) = [0;c_2,c_3,\dots]$ (for details, see \cite[Chpt. 3]{Einsiedler2010}, and Appendix \ref{cf::section}).
	
	Define $B \subset [0,1]$ to be the set of those $\beta \in [0,1]$ whose expansion is of the form $\b = [0;b_1,b_2,\dots,b_l,*,*,\dots]$. In simpler terms, $B$ is an interval with endpoints $[0;b_1,b_2,\dots,b_l]$ and $[0;b_1,b_2,\dots,b_l+1]$. Clearly, $\mu(B) > 0$.
	
	By the poitwise ergodic theorem, we have for all $\a$ but a set of zero measure that
	$$\densNat(J) 
		= \lim_{N \to \infty} \frac{1}{N} \sum_{n=1}^{N} 1_B(T^n (2\a)) = \int 1_B d\mu = \mu(B). \qedhere
	$$
\end{proof}

\begin{proof}[Proof of Proposition \ref{neg::cf::cor:rand->not-bot-var}, case $\e(n) \leq \e_0 < \frac{1}{4}$] \mbox{}

	Let $A$ be a large, odd integer, to be specified in the course of the proof. For almost all choices of $\a$, the sequence $(A,A,A)$ appears infinitely often in $(a_i)_{i=1}^\infty$, i.e. there exists an infinite set $J$ such that for each $j \in J$ we have $a_{j+1} = a_{j+2} = a_{j+3} = A$.
	
	For each $j \in J$ we claim that we can find $i = i(j) \in \{j,j+1,j+2\}$ such that $p_i,q_i$ are both odd.
	
	If $p_j,q_j$ are odd, we are done. Else, because $\gcd(p_j,q_j) =1$, precisely one of $p_j,q_j$ is even; suppose for concreteness that $p_j$ is even. If $p_{j+1},q_{j+1}$ are both odd we are done. Otherwise, $q_{j+1}$ is odd, because $p_j,p_{j+1}$ cannot both be even. We have the recursive relation $p_{j+2} = a_{j+2}p_{j+1} + p_{j} = A p_{j+1} + p_j$, so $p_{j+2}$ is odd. By similar argument, $q_{j+2}$ is odd, so we are done.
	
For any $j \in J$, take $N_j = q_{i(j)}$ and $\gamma_j = \delta_{i(j)}$. By construction, $N_j$ is odd and we have
$$
	N_j \a = \frac{p_{i(j)}}{2} + \frac{\gamma_j}{2 N_j}.
$$

We are now in position to apply  Lemma \ref{neg::prop:obst-eps-adv}.
It follows that $N_j \not \in \A\e\a$, provided that $\e(n) \leq \e_1$ for all $n$, where $\e_1 < \frac{1}{4} - \frac{1}{4}\abs{\gamma_j}$ for all $j$. We know that $\abs{\gamma_j} < \frac{1}{a_{i(j)+1}} = \frac{1}{A}$, so it will suffice to ensure that $\e_0 < \frac{1}{4} - \frac{1}{4A}$, which can be accomplished by choosing sufficiently large $A$.
\end{proof}

We will next deal with the situation when $\e(n) \to 0$. Surprisingly, this is more difficult, because care is needed to ensure that the irrationality condition in Lemma \ref{neg::prop:obst-var} is satisfied. 

We  need a preliminary lemma about estimation of the error term $\de_i$ based on the knowledge of a limited number of continued fraction coefficients. Recall that $\de_i$ is related to $\a$ by $2\a = \frac{p_i}{q_i} + \frac{\delta_i}{q_i^2}$.

\begin{lemma}\label{neg::lem:partial-approximation}
	Let $l$ be a positive integer. There exists a function $\tilde\delta_{l}\colon \NN_{\geq 1}^{2l+1} \to \RR$ such that for any $n > l$ we have
	\begin{equation}
	\abs{ \de_n - (-1)^n\tilde{\de}_l\bra{ (a_i)_{i=n-l}^{n+l}} } \ll 2^{-l/2},
	\end{equation}
	where the implicit constant is absolute.

\end{lemma}
\begin{proof}
	Recall that $\delta_n = q_n^2( 2\a -  \frac{p_n}{q_n})$. Using standard fact about continued fractions, putting $\rho_n = [a_n;a_{n+1},\dots]$ we may write
\begin{align*}
	\delta_n &= 
	q_n^2 \bra{ 
	\frac{p_{n-1} \rho_n + p_{n-2} }{ q_{n-1} \rho_n + q_{n-2} } - 
	\frac{p_{n-1} a_n + p_{n-2} }{ q_{n-1} a_n + q_{n-2} } } 
	\\&=  \frac{ (-1)^n \bra{  q_{n-1} a_n + q_{n-2} }^2 (\rho_{n}-a_n) }
	{ \bra{ q_{n-1} a_n + q_{n-2} } \bra{q_{n-1} \rho_n + q_{n-2}} }.
 \end{align*}
	Recalling that $ \frac{q_{n-2}}{q_{n-1}} = [0;a_{n-1},a_{n-2},\dots] := \lambda_n$ we may simplify the above formula to
$$
	\delta_n =  (-1)^n (\rho_{n}-a_n) 
	\frac{ {  a_n + \lambda_{n} }  }
	{ {  \rho_n + \lambda_{n} } }.
$$
Putting $\tilde \rho = \tilde \rho\bra{ (a_i)_{i=n-l}^{n+l}} := [a_n;a_{n+1},\dots, a_{n+l} ]$ and $\tilde \lambda = \tilde \lambda \bra{ (a_i)_{i=n-l}^{n+l}}:= [0;a_{n-1},a_{n-2},\dots,a_{n-l}]$ we have $\rho_n = \tilde \rho + O(2^{-n/2})$ and $\lambda_n = \tilde \rho + O(2^{-n/2})$. It remains to put 
$$\tilde\delta_l = \tilde \delta_l \bra{ (a_i)_{i=n-l}^{n+l}} := 
	(-1)^n (\tilde \rho-a_n) 
	\frac{ {  a_n + \tilde \lambda }  }
	{ {  \tilde \rho + \tilde \lambda } }. \qedhere
$$
\end{proof}

\begin{proof}[Proof of Proposition \ref{neg::cf::cor:rand->not-bot-var}, case $\e(n) \to 0$]\mbox{}

Using Proposition \ref{neg::prop:normality}, for almost all choices of $\a$, we may find arbitrarily long strings of $1$'s in the expansion $(a_i)_{i=1}^\infty$. More precisely, there exists an infinite set $J$ and a sequence $l(j)$, $j \in J$ with $l(j) \to \infty$ as $J \ni j \to \infty$, such that for each $j \in J$ and for each $\abs{t} \leq l(j)$ we have $a_{j+t} = 1$.

Repeating the argument from the proof of the same proposition in the pointwise bounded case, we may find for each $j \in J$ and index $i = i(j) \in \{j,j+1,j+2\}$ such that $p_i,q_i$ are both odd. Without loss of generality we may assume that $i(j) = j$, i.e. that $p_j,q_j$ are both odd for $j \in J$.

Let us put $N_j := q_j$ and $\gamma_j := \delta_j$, so that 
$$
	N_j \a \equiv \frac{p_{i(j)}}{2} + \frac{\gamma_j}{2 N_j} \pmod{1}.
$$
Applying Lemma \ref{neg::prop:obst-var}, we conclude that either $N_j \not \in \A\e\a$ for infinitely many $j$, or for each limit point $\gamma$ of $\gamma_j$ there exists $n$ such that $\gamma + 2 n^2 \a \in \ZZ$.

Passing to a subsequence, we may assume without loss of generality that $\gamma_j$ converges. Using Lemma \ref{neg::lem:partial-approximation} we may identify $\gamma := \lim_{j \to \infty} \gamma_j$: 
$$
	\gamma = \pm \frac{1}{\varphi}\cdot \frac{1+\frac{1}{\varphi}}{\varphi + \frac{1}{\varphi}} = \pm \frac{1}{\sqrt{5}},
$$
where $\varphi = \frac{1+\sqrt{5}}{2}= [1;1,1,\dots]$. 

There are two cases to consider, depending on whether $\a$ and $\gamma$ are affinely independent over $\ZZ$. If they are, then we are done by Lemma \ref{neg::prop:obst-var}. Otherwise, $\a \in \QQ[\sqrt{5}]$. However, we can exclude this case, since $\QQ[\sqrt{5}]$ has measure $0$ (alternatively, we can apply Proposition \ref{neg::sqrt::prop:main}).
\end{proof}

\begin{remark}
	It is tempting to try to repeat the argument for the case $\e(n) \leq \e_0$ in the case $\e(n) \to 0$. Arguing along these lines, one can find a sequence of odd integers $N_j$ such that $N_j \a \equiv \frac{p_{i(j)}}{2} + \frac{\gamma_j}{2 N_j} \pmod{1}$ with $p_{i(j)}$ odd and $\gamma_j \to 0$. Lemma \ref{neg::prop:obst-var} would be applicable with $\gamma = 0$. We may conclude (inspecting the proof of Lemma \ref{neg::prop:obst-var}) that for sufficiently large $j$, the only possible representation of $N_j$ as a member of $2\A\e\a$ is $N_j = N_j + 0$. However, $0 \in \A\e\a$, and we cannot exclude the possibility that $N_j \in \A\e\a$, hence the need for a more involved argument.
\end{remark}


\setcounter{section}{1}
\section{Largeness and equidistribution}
\label{pos::section}	

In the previous Section \ref{neg::section} we have seen that the sets $\A{\a}{\e}$ (as defined in \ref{neg::eq::def-of-A}) usually are not \bases\ of order $2$. Our goal in this section is to show that the sets $2 {} \A{\a}{\e}$ nevertheless tend to be quite sizeable. For the convenience of the reader we recall the statements of our main theorem, stated in the introduction. Our first result deals with density, and applies in a fairly general situation.

\begin{theorem*}[\ref{intro::thm:A:1}, reiterated]
	Let $\a \in \RR \setminus \QQ$. Then, there exists a decreasing sequence $\e_\a(n) \to 0$ such that  $\A{\e}{\a}$ is an almost \bot, provided that $\e(n) \geq \e_\a(n)$ for all $n$. 	
\end{theorem*} 

We will also prove a more surprising result, which shows that results from Section \ref{neg::section} cannot be generalised to all $\a$.

\begin{theorem*}[\ref{intro::thm:A:3}, reiterated]
	There exist an uncountable set $E \subset \RR$ such that for any $\a \in E$, there exists a decreasing sequence $\e_\a(n) \to 0$ such that  $\A{\e}{\a}$ is a \bot, provided that $\e(n) \geq \e_\a(n)$ for all $n$. 
\end{theorem*} 

	As the reader will have noticed, because of the monotonicity of the family $\A{\e}{\a}$ with respect to $\e$, in both theorems there is no loss of generality in assuming that $\e(n) = \e_\a(n)$ for all $n$.

\subsection{Equidistribution and quantitative rationality}\label{pos::equi::section}
 
In Section \ref{neg::section}, specifically in Lemmas \ref{neg::prop:obst-eps-adv} and \ref{neg::prop:obst-var}, we have identified a class of obstructions to $\A{\e}{\a}$ being a \bot. Namely,  we found sufficient conditions for a large integer $N$ to fail to belong to $2 {} \A{\e}{\a}$. 

Here, our first goal is to prove that these obstructions are essentially the only possible ones. We obtain two subtly different results, which can be construed as partial converses to Lemmas \ref{neg::prop:obst-eps-adv} and \ref{neg::prop:obst-var}. Because $\fpa{\cdot}$ is always at most $\frac 12$, we implicitly assume that $\e_0 \leq \frac 12$ in what follows.

\begin{lemma}\label{pos::equi::prop:nequi->ratio}
	There exists a constant $C$ such that the following is true. Let $\a \in \RR \setminus \QQ$, let $\e(n)$ be pointwise bounded as $\e(n) \geq \e_0$, and suppose that $N \not \in 2 {} \A{\e}{\a}$. Then there exists $0< k \leq 1/\e_0^C$, such that
	\begin{equation}
	\fpa{k N\a } \leq \frac{1}{N \e_0^C} .
	 \label{pos::eq:021}
\end{equation}
\end{lemma}

\begin{lemma}\label{pos::equi::prop:nequi->ratio-exact}
	Let $\a \in \RR \setminus \QQ$ and $\e_1 > 0$. For any $\e_0 > \e_1 $ there exists $N_0 = N_0(\a,\e_1,\e_0)$ such that following is true for $N \geq N_0$. Suppose that $\e(n) \geq \e_0$ for all $n$ and that $N \not \in 2 {} \A{\e}{\a}$. Then $N$ is odd and there exist $m,k \in \NN$ such that $2 \mid k$, $\gcd(m,k)=1$, and $\gamma \in \RR$ such that
	\begin{equation}
	 N \a = \frac{m}{k} + \frac{\gamma}{k N}, \qquad \frac{1 - \abs{\gamma}}{2k} > \e_1.
	 \label{pos::eq:022}
	 \end{equation}
\end{lemma}

We pause to describe the difference between these two results. Both state that if $N \not \in 2 {} \A{\e}{\a}$, then $N \a$ is well approximated by a rational with small denominator. In \ref{pos::equi::prop:nequi->ratio}, the quality of approximation is worse, but no additional assumptions are imposed on $N$. On the other hand, in \ref{pos::equi::prop:nequi->ratio-exact} we obtain detailed information, but we need to restrict to sufficiently large $N$. In particular, \ref{pos::equi::prop:nequi->ratio} is non-vacuous in the regime $\e \sim 1/N^\delta$ with $\delta$ sufficiently small.

The first step in order to prove Lemmas \ref{pos::equi::prop:nequi->ratio} and \ref{pos::equi::prop:nequi->ratio} is to reduce the problem of representing $N$ as an element of $2 {} \A{\e}{\a}$ to an equidistribution statement about an orbit on the torus.

\begin{observation*}\label{pos::equi::obs:return<->2xA}
	Fix $\a$, let $\e(n) = \e_0$ be constant, and let $N \in \NN$. Then $N \in 2 {} \A{\e}{\a}$ if and only if the quadratic orbit 
	$$x_n := \left(n^2 \a, (N-n)^2 \a \right) \in \TT^2$$
	enters the set $(-\e_0,\e_0)^2 \subset \TT^2$ at time up to $N$, i.e. if there exists $0 < n \leq N$ such that $x_n \in (-\e_0,\e_0)^2$.
\end{observation*}

For a sequence $(x_n)_{n=1}^\infty \in X$ of points in a compact metric space endowed with a probability measure $\mu$, we shall say, following terminology e.g. in \cite{Green2007}, that $x_n$ is \emph{$(\delta,N)$-equidistributed} if and only if for each $f \in \operatorname{Lip}(X;\RR)$ we have
	$$ \abs{\EE_{n \leq N} f(x_n) - \int_X f d \mu} \leq \de \normLip{f}.$$	
Here, $\normLip{f}$ denotes the Lipschitz norm $\normLip{f} = \sup_{x,y \in X} \frac{\abs{f(x)-f(y)}}{d_X(x,y)}$, and $\operatorname{Lip}(X;\RR) \subset \cC(X;R)$ denotes the space of those $f$ with $\normLip{f} < \infty$.

Although we are ultimately interested in density, equidistribution turns out to be easier to work with. Of course, not every dense sequence is equidistributed. However, equidistribution implies density, if we allow for a slight change in the parameters. The following observation is elementary.

\begin{observation*}\label{pos::equi::obs:dens<->equi}
	Suppose that $X$ is a $d$-dimensional compact smooth manifold equipped with a Riemannian metric and with a measure $\mu$ arising from a volume form. Then there exists a constant $c >0$, such that the following is true. Let $\delta > 0$, and suppose that a  sequence $x_n$ is $(c \delta^d, N)$-equidistributed. Then for any $x \in X$, there exists $n$ such that $d_X(x,x_n) < \delta$.
\end{observation*}

It is a classical result of Weyl that lack of equidistribution of a polynomial orbit on the torus can always be explained by a rational obstruction. We have the following classical theorem (see \cite[Thm. 1.4]{Einsiedler2010}).

\begin{theorem}[Weyl equidistribution]\label{pos::equi::thm:Weyl}
	For any $d$ there exist a family of constants $N_0(p,\delta)$ such that the following is true.
	
	Let $p(n) = (p_i(n))_{i=1}^d$ be a polynomial sequence in $\TT^d$. Suppose that $p(n)$ is not $(\de,N)$-equidistributed. Then either $N < N_0(p,\de)$, or there exists $k \in \ZZ^d \setminus \{0\}$ such that if $\sum_i k_i p_i = \sum_j \a_j n^j$, then $\a_j \in \ZZ$ for all $j$. 
\end{theorem} 

We will need a quantitative version of the above theorem. The following result is a special case of Theorem 1.16 in \cite{Green2007}.

\begin{theorem}\label{pos::equi::thm:GT}
	For any $d,r$ there exists a constant $C$ such that the following is true.
	
	Let $p(n) = (p_i(n))_{i=1}^d$ be a polynomial sequence in $\TT^d$ with $\deg p = r$. Suppose that $p(n)$ is not $(\de,N)$-equidistributed. Then there exists $k \in \ZZ^d \setminus \{0\}$ such that $k_i \ll \frac{1}{\de^C}$ and if we write $\sum_i k_i p_i = \sum_j \a_j n^j$ then $\fpa{\a_j} \ll \frac{1}{N^j \de^C}$. 
\end{theorem} 

We are now ready to prove the main results in this section. 

\begin{proof}[Proof of Lemma \ref{pos::equi::prop:nequi->ratio}]\mbox{}

	Since $N \not \in 2 {} \A{\e}{\a}$, the orbit $(n^2\a,(N-n)^2\a)$ misses $(-\e_0,\e_0)^2$ up to time $N$. It follows that $(n^2\a,2 N n \a)$ fails to be $(c\e_0^2,N)$-equidistributed with $c > 0$.

By the characterisation of equidistribution in Theorem \ref{pos::equi::thm:GT}, it follows that there is a universal constant $C$ such that we can find $k_1,k_2$ with $\abs{k_i} \ll \frac{1}{\e_0^C}$, $(k_1,k_2) \neq (0,0)$, such that $\fpa{k_1 \a} \ll \frac{1}{N^2 \e_0^C}$ and  $\fpa{k_2 N \a} \ll \frac{1}{N\e_0^C}$. Hence, for $i=1,2$ we have $\fpa{k_i N \a} \ll \frac{1}{N\e_0^C}$, and since both of $k_i$ cannot be $0$, the claim follows.
\end{proof}

\begin{proof}[Proof of Lemma \ref{pos::equi::prop:nequi->ratio-exact}] \mbox{}

	It follows from the above Lemma \ref{pos::equi::prop:nequi->ratio} that there exists $K = O_{\e_0}(1)$ such that $\fpa{k N \a} = O_{\e_0}(1/N)$ for some $k \in [K]$. Possibly replacing $k$ with one of its divisors, we can therefore write:
	$$ N \a = \frac{m}{k} + \frac{\gamma}{k N},$$
	where $\gcd(m,k)=1$, $\abs{\gamma} \leq G$ and $G = O_{\e_0}(1)$ is a constant. 
	
	Let $\de > 0$ be a small number to be determined later and let $M = M(\de,\a)$ be such that $(n^2\a \Mod{1})$ intersects any interval of length $\delta$ as $n$ ranges over any progression $P = n_0 + l [M']$ with length $M' \geq M$ and step $l \leq K$. We know that such $M$ exists, for instance by Theorem \ref{pos::equi::thm:GT}. 
	
	Note that we have for any $n \in \NN$ we have
	$$ \left( n^2 \a, (N-n)^2 \a\right)	= 
	\left( n^2 \a, n^2 \a + \frac{(N-2n)m}{k} + \frac{N-2n}{N}\frac{\gamma}{k} \right).
	$$
	Thus, the values $(n^2 \a \Mod{1})$ and $((N-n)^2\a \Mod{1})$ depend only on $(n^2 \a \Mod{1})$, $(N-2n \Mod{k})$ and $\frac{N-2n}{N}$. If $k$ is even, then for any choice of $n$, $(N-2n \Mod{k})$ has the same parity as $N$, and obviously $\frac{N-2n}{N} \in [-1,+1]$. These turn out to be essentially the only restrictions.
	\begin{observation}\label{pos::equi::in-obs@nequi->ratio-exact}
		Let $\tau \in \TT,\ b \in [k],\ x \in [-1,1]$, and if $k$ is even, assume additionally that $b \equiv N \pmod{2}$. Then there exists some $n \in [N]$ such that $\fpa{ n^2 \a - \tau} \leq \de$, $(N-2n)m \Mod{k} = b$ and $\abs{ \frac{N-2n}{N} - x} \leq \frac{2 K M}{N}$, provided that $N > 2 K M$.
	\end{observation}
	\begin{proof}
	We can pick $n_0$ such that $\abs{ \frac{N-2n_0}{N} - x} \leq \frac{2}{N}$. Next, we can pick $n_1$ with $\abs{n_0-n_1} \leq k$ such that $(N-2n_1)m \equiv b \pmod{k}$ and $\abs{ \frac{N-2n_1}{N} - x} \leq \frac{2 K}{N}$.
	
	Let $P = n_2 + k[M]$ be an progression of length $M$, step $k$, containing $n_1$, and contained in $[N]$. For $n \in P$ we have $\abs{ \frac{N-2n}{N} - x} \leq \frac{2 K M}{N}$ and $(N-2n)m \equiv b \pmod{k}$. For at least one of these values, we have $\fpa{n^2 \a - \tau} \leq \de$.
	\end{proof}	
	
	If $N$ is even or $k$ is odd, then taking $\tau = 0,\ b = 0,\ x = 0$ and setting $\de = \e_0/2$ we find some $n \in [N]$ that $\fpa{n^2\a} \leq \e_0/2 $ and $\fpa{(N-n)^2\a} \leq \e_0/2 + {2MKG}/{N}$, unless $N \leq 2K M$. Since  $N \not \in 2 {} \A{\e}{\a}$, this situation is only possible if $N \ll MKG/\e_0$. Hence, we may assume that $N$ is odd, $k$ is even.
	
	If $\abs{\gamma} \geq 1$, then taking $\tau = 0,\ b = 1,\ x = \frac{1}{\gamma}$, $\de = \e_0/2$ we again find some $n \in [N]$ that $\fpa{n^2\a} \leq \e_0/2$ and $\fpa{(N-n)^2\a} \leq \e_0/2 +  2MKG/N$, which leads to a contradiction, unless $N \ll MKG/\e_0$. Hence, we may assume this is not the case.
	\newcommand{\sgn}{\operatorname{sgn}}
	
	Finally, let us take $b = 1,\ x = - \sgn \gamma$ and $\tau = -\frac{1}{2k}\left(1 - \abs{\gamma}\right)$ and $\de = (\e_0-\e_1)/2$. Then for some $n \in [N]$ we have (assuming $N \geq KM$)
\begin{align*}
	 \e_0 &\leq \max\left( \fpa{n^2\a} , \fpa{(N-n)^2\a} \right) 
	 \\&\leq \frac{1 - \abs{\gamma}}{2k} + \frac{\e_0-\e_1}{2} + 2\frac{MKG}{N}.	
\end{align*}
	If it holds that $\frac{1 - \abs{\gamma}}{2k} \leq \e_1$, then the above implies that $N \ll MKG/(\e_0 - \e_1)$. Otherwise, the decomposition $N\a = \frac{m}{k} + \frac{\gamma}{k N} $ obtained earlier satisfies the condition $\frac{1 - \abs{\gamma}}{2k} > \e_1$, and we have that $k$ is even and $N$ is odd from previous considerations.
\end{proof}

\subsection{Almost \bases\ of order 2}\label{pos::2::section}

With tools introduced in \ref{pos::equi::section}, we are ready to prove the first of the two main result of this section, of which Theorem \ref{intro::thm:A:1} is a special case.

	To formulate the theorem, we need an additional a piece of notation. For real $\a$, the \emph{irrationality measure of $\a$}, denoted $\mu(\a)$, is the smallest value $\mu$ such that for any $\delta > 0$ we have $$\abs{\a - \frac{p}{q}} \geq \frac{c}{q^{\mu + \delta}}$$ for any $p,q$ with $\a \neq \frac{p}{q}$, where $c = c(\a,\delta,\mu) > 0$ is a constant independent of $p$ and $q$. If no such $\mu$ exists, then $\mu(\a) = \infty$. We also recall that $\a$ is said to be \emph{badly approximable}, if it holds that $\abs{\a - \frac{p}{q}} \geq \frac{c}{q^2}$ for any integers $p,q$, where $c = c(\a) > 0$. 

For $\alpha \in \QQ$ we have (somewhat artificially) $\mu(\a) = 1$, and for any other $\a$, $\mu(\a) \geq 2$. For almost all (with respect to Lebesgue measure) $\a$, we have $\mu(\a) = 2$. Specifically, this holds for algebraic numbers, which is a celebrated result due to Roth \cite{Roth2010}. 

\begin{theorem}\label{pos::equi::thm}
	Let $\a \in \RR \setminus \QQ$. Then, there exists a decreasing sequence $\e_\a(n) \to 0$ such that for any $\e$ with $\e(n) \geq \e_\a(n)$ for all $n$, the set $\A\e\a$ is an almost \base\ of order $2$. 
	
	Moreover, if $\mu(\a) < \infty$, then the assumption $\e(n) \geq \e_\a(n)$ can be replaced with the assumption that $\frac{\log(1/\ef(n))}{\log n} \to 0$. In this case, we additionally have the estimate
	$$\abs{ [T] \setminus 2 {} \A \e\a } \ll T^{1 -c},$$
	where the constant $c > 0$ depends only of $\a$.
	
	Finally, if $\a$ is badly approximable, and $\e(n) \geq \e_0 > 0$ for all $n$ we have a sharper estimate 
		$$\abs{ [T] \setminus 2 {} \A \e\a } \ll \log T,$$
	where the implicit constant depends only on $\a$ and $\e_0$.
\end{theorem}

We begin by proving a technical proposition which describes local sparsity of the complement of $2 {} \A{\e}{\a}$. We wish to point out that this is a slightly stronger type of statement than Theorem \ref{pos::equi::thm}, since even sets with extremely slow asymptotic growth can contain many consecutive elements. 

\begin{proposition}\label{pos::prop:large-gaps}
	There exists a constant $C$ such that the following is true. Let $\a \in \RR \setminus \QQ$ and suppose that $\e(n) \geq \e_0 > 0$ is pointwise bounded from below. Suppose that $N,N' \not \in 2 {} \A{\e}{\a}$ and $N'>N$. Then the following statements hold.
	\begin{enumerate}
	\item If $\a$ is badly approximable, then $N'-N \gg N \e_0^C$.
	\item If $\mu(\a) < \infty$ and $\tau > \frac{\mu(\a) - 2}{\mu(a) - 1}$, then $N' - N \gg_{\tau} N^{1-\tau} \e^C_0$.
	\item If $\mu(\a) = \infty$ then $N'-N \gg \e^C_0 \omega(N \e^C_0)$, where $\omega(t) \to \infty$ as $t \to \infty$.
	\end{enumerate}

\end{proposition}
\begin{proof}
	Since $N,N' \not \in 2 {} \A{\e}{\a}$, it follows from Lemma \ref{pos::equi::prop:nequi->ratio} that, there are $k,k' \leq 1/\e_0^{C}$ such that $\fpa{kN\a}, \fpa{k'N'\a} \leq \frac{1}{N \e_0^C}$, where $C$ is a universal constant. 

	Let $Q_0(\de)$ denote the least positive integer such that $\fpa{Q_0(\de)\a} \leq \de$. For any irrational $\a$ we have $Q_0(\de) \to \infty$ as $\de \to 0$. Moreover, if $\mu > \mu(\a)$ and $\frac 1{1-\tau} = \mu - 1$, (resp. if $\mu = 2$ and $\tau = 0$ if $\a$ is badly approximable), then $Q_0(\de) \gg 1/\de^{1 - \tau}$.

	Let $L := N-N'$ and let $m := kk' \leq 1/\e_0^{2C}$. We have
	$$ \fpa{ mL\a} \leq k \fpa{ k'N'\a} + k' \fpa{ kN \a} \leq \frac{2}{N\e_0^{2C}},$$
	and as a consequence we have 
		$$ \frac{L}{\e_0^{2C}} \geq  Q_0 \left(\frac{2}{N\e_0^{2C}} \right) \gg \left( N\e_0^{2C} \right)^{1-\tau}  .$$
	In each case, this easily leads to the sought bound.
\end{proof}
	
\begin{proof}[Proof of Theorem \ref{pos::equi::thm}]\mbox{}

	We may assume without loss of generality that $\e(2n) \geq \frac{1}{2} \e(n)$, and that $\e(n)$ is non-increasing. Let us denote $\cN := \NN \setminus 2 {} \A{\e}{\a}$ and enumerate $\cN = \{N_i\}_{i  = 1}^\infty$ so that $N_{i+1} > N_i$.

	Our first aim is to show that the sequence $N_i$ increases rapidly enough. Let us take any $i$. If $N_{i+1} \geq 2 N_i$, then we have sufficiently good lower bound for $N_{i+1}$, so suppose that this is not the case. 
	Since $N_{i+1},N_i \not \in 2 {} \A{\e(N_{i+1})}{\a}$, we may apply Proposition \ref{pos::prop:large-gaps} to conclude that
	\begin{align}
		N_{i+1} - N_i \gg
		\begin{cases}
		N_i \e(N_i)^C  & \text{ if $\a$ b. approx., with $C > 0$}, \\
		N_i^c \e(N_i)^C &\text{ if $\mu(\a) < \infty$, with $C,c > 0$},\\
		\e(N_i)^C \omega(N_i \e(N_i)^C) &\text{ else, with $\omega(t) \to \infty$, $C > 0$.}
		\end{cases}
	\label{pos::equi::eq:032}
	\end{align}

	In the general case, we have $N_{i+1} - N_i \to \infty$, as $i \to \infty$, provided that $\e(n)^C \omega(n \e(n)^C) \to 0$, as $n \to \infty$. The latter condition is satisfied for constant $\e(n)$, and hence also for some slowly decaying $\e_\a(n) \to 0$. Lack of control on $\omega$ makes it impossible to say anything more explicit about $\e_\a$.
	
	In the case when $\mu(\a) < \infty$, let us assume that $\e(n) \gg 1/N^\delta$, where $\delta $ is small enough. We have $N_{i+1} \gg N_i + c_2 N_i^{c_1}$ with $c_1,c_2>0$. A simple inductive argument shows that in this case we have $N_i \gg i^{1+c}$ for some $c>0$. Hence, $\abs{ [T]\cap \cN} \ll T^{\frac{1}{1+c}}$, proving the sought bound.
	
	In the case when $\a$ is badly approximable and $\e(n) \geq \e_0$ is bounded pointwise, we have $N_{i+1} \gg N_i(1+\e_0^C)$ with some constant $C$. It follows by a simple inductive argument that $\log N_i \gg \log i$. In particular, $\abs{ [T]\cap \cN} \ll \log T$.
\end{proof}

\begin{remark*} 
	Essentially the same argument leads to a result in higher dimension. More precisely, if $\a \in \RR^r$ and we define
	$$\cA = \set{n \in \NN}{ \fpa{ n^2 \a_i} < \e \text{ for } i=1,\dots,r},$$ where $\e > 0 $ is constant, then one can show that $2{} \cA$ has density $1$, provided that $k \cdot \a = \sum_i k_i \a_i$ is irrational for all $k \in \ZZ^r \setminus \{0\}$.
\end{remark*}

\subsection{Exceptional values of $\a$}

We have seen in Section \ref{neg::section} that the sets $\A{\e}{\a}$ tend not to be \bases\ of order $2$. The main result of this section shows that such statements do not generalise to \emph{all} values of $\a$: we can find values of $\a$ such that the set $\A\e\a$ is a \bot\ as soon as $\e(n) \geq \e_0 > 0$. 

 For such values of $\a$ we also have that $\A{\e}{\a}$ is a \bot\ for some $\ef(n) \to 0$. However, we have little control over the rate of convergence, so we do not pursue this issue further.

Our approach amounts to carefully preventing the conditions \eqref{pos::eq:022} in Lemma \ref{pos::equi::prop:nequi->ratio-exact} from being satisfied. The crucial step is establishing some control over all good approximations of $\a$.

\mbox{}\\
Throughout this section, we work with $\a \in (0,1) \setminus \QQ$, we let $a_i$ denote the digits in the continued fraction expansion of $\a$:
\begin{align}
	\a = [a_0;a_1,a_2,\dots] = a_0+\cfrac{1}{a_1+\cfrac{1}{a_2+\cdots}}
\end{align}
(note difference with usage in Section \ref{neg::section}), and $\frac{p_i}{q_i}$ denote the rational approximations of $\a$ arising from the truncated continued fractions: $\frac{p_i}{q_i} = [a_0;a_1,a_2,\dots, a_i]$.

Recall that $\nu_p(a)$ denotes the largest power of the prime $p$ dividing $a$. We will be interested in $\a$ satisfying the following conditions:
\begin{align}
	\nu_2(q_i) &\tendsto{}  \infty && \text{as $i \to \infty$ through even numbers,} \label{pos::except::cond:p|q-i-1} \\
	\nu_p(q_i)  &\tendsto{}  \infty && \text{as $i \to \infty$ through odd numbers, for $p$ odd prime.} 
		\label{pos::except::cond:p|q-i-2}
\end{align}

	We observe that if $a_i$ obey the conditions (\ref{pos::except::cond:p|q-i-1}, \ref{pos::except::cond:p|q-i-2}), then they also obey the following conditions: 
	\begin{align}
	\nu_2(a_i) &\tendsto{}  \infty && \text{as $i \to \infty$ through even numbers,} 	\label{pos::except::cond:p|a-i-1}
\\
	\nu_p(a_i) &\tendsto{}  \infty && \text{as $i \to \infty$ through odd numbers, for $p$ odd prime.} 
	\label{pos::except::cond:p|a-i-2}
\end{align}
This is an easy consequence the facts that $a_{i} = \frac{q_{i} - q_{i-2}}{ q_{i-1}}$ (\ref{cf:recursive}) and $\gcd(q_{i},q_{i-1}) = 1$ (\ref{cf:Delta}). 

The following observation ensures that our considerations are not vacuous. It is not difficult, but the proof is slightly mundane.
		
\begin{observation}\label{pos::except::obs:exist-a}
	There exist uncountably many $\a$ such that the conditions (\ref{pos::except::cond:p|q-i-1}, \ref{pos::except::cond:p|q-i-2}) (and hence also (\ref{pos::except::cond:p|a-i-1}, \ref{pos::except::cond:p|a-i-2})) are satisfied. Moreover, for any $h_i \in \NN$ with $h_i \to \infty$ as $i \to \infty$, we can additionally require that $a_i \leq h_i$ for all $i$.
\end{observation}
\begin{proof}		
	Let $p$ be a prime and $k_i^{(p)}$ a sequence of integers with $k_i^{(p)} \to \infty$. We will construct a sequence $a_i^{(p)}$ such that $a_i^{(p)} \leq p^{k_i^{(p)}}$ and an associated sequence $q_i^{(p)}$ (related to $a_i^{(p)}$ by $r^{(p)}_i/q^{(p)}_i = [a^{(p)}_0;a^{(p)}_1,a^{(p)}_2,\dots, a^{(p)}_i]$ for some $r^{(p)}_i$), in such a way that $\nu_p(q^{(p)}_i) \geq k_i^{(p)}$ for sufficiently large $i$. Once this is done, we choose some sequences $k_{i}^{(p)}$ with $k_{i}^{(p)}  \to \infty$ such that $\prod_p p^{k_i^{(p)}} \leq h_i$, and define $a_i \leq \prod_p p^{k_i^{(p)}}$ by requiring that $a_i \equiv a_i^{(p)} \pmod{ p^{k_i^{(p)}}}$. It is clear that thus defined sequence $a_i$ satisfies (\ref{pos::except::cond:p|q-i-1}, \ref{pos::except::cond:p|q-i-2}).
	
	To construct $a_i^{(p)}$ we proceed by induction. We restrict to the case when $p \neq 2$, the case $p = 2$ is fully analogous. We may assign arbitrary values $a_i^{(p)}$ for a number of small values of $i$. In particular, in the construction we may assume that $i$ is large enough that $k_{i}^{(p)} \geq 1$. Suppose that $a^{(p)}_1,\dots,a^{(p)}_i$ have been constructed for some $i \equiv 1 \pmod{2}$. Since at most one of $q_i^{(p)},q_{i+1}^{(p)}$ is divisible by $p$, we can choose $a_{i+1}^{(p)} \leq p \leq p^{k_{i+1}^{(p)}}$ so that that $q_{i+1}^{(p)} = a_{i+1}^{(p)} q_{i}^{(p)} + q_{i-1}^{(p)} \not \equiv 0 \pmod{p}$. Next, since $p \nmid q_{i+1}^{(p)}$, we may choose $a_{i+2}^{(p)} \leq p^{k_{i+2}^{(p)}}$ so that $q_{i+2}^{(p)} = a_{i+2}^{(p)} q_{i+1}^{(p)} + q_{i}^{(p)}  \equiv 0 \pmod{p^{k_{i+2}^{(p)}}}$. This finishes the inductive step. It follows from the construction that $\nu_p(q_{i}) \geq p^{k_{i}^{(p)}}$ for all but finitely many $i$, so the sequence satisfies the required conditions.
	
	To show that the number of possible choices of $\a$ is uncountable, we notice that different choices of the sequences $k_i^{(p)}$ produce different $\a$. Since there are uncountably many choices for $k_i^{(p)}$, there are also uncountably many choices of $\a$.
\end{proof}

Theorem \ref{intro::thm:A:3} is an immediate consequence of the following slightly more technical result, paired with Observation \ref{pos::except::obs:exist-a}.

\begin{proposition}\label{pos::except::prop:cond->A-is-b.o.2}
	Suppose that for some $\a$, the conditions (\ref{pos::except::cond:p|q-i-1}, \ref{pos::except::cond:p|q-i-2}) and  (\ref{pos::except::cond:p|a-i-1}, \ref{pos::except::cond:p|a-i-2}) are satisfied, and that $\frac{\log a_i}{i} \to 0$. Then $\A{\e}{\a}$ is a base of order $2$ provided that $\e(n) \geq \e_0 > 0$ for all $n$.
\end{proposition}
\begin{proof}
	For a proof by contradiction, suppose that there is some $\e_0 > 0$ such that $\A{\e_0}{\a}$ is not a base of order $2$.
	
	By Lemma \ref{pos::equi::prop:nequi->ratio-exact}, for infinitely many odd $N$ there exists $k,m,\gamma$ with $k$ even, $\gcd(m,k) = 1$, such that	
	$$ N\a = \frac{m}{k} + \frac{\gamma}{k N} ,\qquad \frac{1}{2k}(1 - \abs{\gamma}) > \frac{\e_0}{2}.$$
	Because $k$ is automatically bounded by $k \leq \frac{1}{\e_0}$, we may assume that $k$ does not depend on $N$. Moreover, we may assume that $m$ and $N$ are coprime, because a similar relation is satisfied for $N' = \frac{N}{\gcd(N,m)}$, $m' = \frac{m}{\gcd(N,m)}$ and $\gamma' = \frac{\gamma}{\gcd(N,m)^2}$.
	
	Let us fix $N$, but we reserve the right to assume that $N$ is sufficiently large in terms of $\e_0$ and $k$, and take $m,\gamma$ as above. Put $q = k N$ and $p = m$, and let $i$ be the largest index such that $\frac{p}{q}$, lies between $\a$ and $\frac{p_i}{q_i}$, where $q := kN$. We may without loss of generality assume that $i$ is odd so that $\a < \frac{p_{i+2}}{q_{i+2}} < \frac{p}{q} \leq \frac{p_i}{q_i}$. The other case, when $\a > \frac{p_{i+2}}{q_{i+2}} > \frac{p}{q} \geq \frac{p_i}{q_i}$ is fully analogous.
	
	The case when $\frac{p}{q} = \frac{p_i}{q_i}$ is particularly simple. Because $k$ is even, $q_i = q = kN$ is even, and in particular $i$ is even (because of assumption \eqref{pos::except::cond:p|q-i-1}). Since $N$ is odd, $ \nu_2(q_i) = \nu_2(k) $ is bounded.  However, this contradicts condition \ref{pos::except::cond:p|q-i-1}, provided that $N$ (and hence $i$) is sufficiently large. Hence, we may assume that $\frac{p}{q} < \frac{p_i}{q_i}$ 
	
	We now deal with the general $\frac{p}{q}$. We can write $\frac{p}{q} = \frac{a p_{i+1} + b p_i}{a q_{i+1} + b q_i}$ for some coprime $a,b \in \NN$, simply because $ \frac{p_{i+1}}{q_{i+1}} < \frac{p}{q} < \frac{p_{i}}{q_{i}}$. A straightforward computation using \ref{cf:Delta} shows that
	$$\frac{p}{q} - \frac{p_{i+2}}{q_{i+2}} = \frac{ (a_{i+2}b-a)(p_{i+1}q_i - p_i q_{i+2})}{q q_{i+2}} = \frac{\Delta}{q q_{i+2}},$$
	where $\Delta := a_{i+2}b-a \geq 1$. It follows that
	$$  \frac{k \abs\gamma}{q^2} = \frac{\abs\gamma}{kN^2} = \abs{\frac{p}{q} - \a} > \abs{\frac{p}{q} - \frac{p_{i+2}}{q_{i+2}}} = \frac{\Delta}{qq_{i+2}}.$$ 
	Thus, we have the bound
	\begin{equation}
	{\gamma} \geq \frac{\Delta}{k} \frac{q}{q_{i+2}} = \frac{\Delta}{k}\left(b- \frac{\Delta}{q_{i+2}}\right).
	\label{pos::eq:037}
	\end{equation}
	To have some rather crude control on the size of $\Delta$, we note that $q \geq q_i$ (\ref{cf:best-approx}) so
	$$ 1 \geq \abs{\gamma} \geq  \frac{\Delta}{k} \frac{q}{q_{i+2}} \geq \frac{\Delta}{4 k a_{i+2} a_{i+1}},$$
	which leads to $\Delta \leq 4 k a_{i+2}a_{i+1} \ll \left(1+\frac{1}{10}\right)^i$. On the other hand, because of \ref{cf:recursive} we have $q_i \gg \sqrt{2}^{\, i},$
	so if  $N$ (and hence also $i$) is sufficiently large, then we have $\frac{\Delta}{q_{i+2}} < \e_0$. Combining this with previous bounds, we find that
	$$ 1 - 2k \e_0 \geq \abs{\gamma} > \frac{\Delta}{k}\left(b-\e_0\right),$$
	which in particular implies that $1 > \frac{\Delta b}{k}$, provided that $\e_0$ is small enough.

	Let us write $k = k_0k_1$ as a product of a power of $2$ and an odd integer. Recall that we have $k_0k_1 = \frac{q}{N} \mid  a q_{i+1} + b q_{i}$. Assuming that $N$ (and hence $i$) is sufficiently large, and possibly exchanging the order of $k_0,k_1$, we conclude from conditions (\ref{pos::except::cond:p|q-i-1}, \ref{pos::except::cond:p|q-i-2}) that $k_0 \mid q_i,\ \gcd(k_0, q_{i+1}) =1$ and $k_1 \mid q_{i+1},\ \gcd(k_1, q_{i})=1$. The divisibility condition $k_0 \mid  a q_{i+1} + b q_{i}$ reduces to $k_0 \mid a$ and likewise $k_1 \mid  a q_{i+1} + b q_{i}$ reduces to $k_1 \mid b$.  
	
	Clearly, $k_1 \mid b$ implies that $b \geq k_1$. From $k_0 \mid a$, we have $ k_0 \mid \Delta = a_{i+2} b - a$ because of conditions (\ref{pos::except::cond:p|a-i-1}, \ref{pos::except::cond:p|a-i-2}). Consequently, $\Delta \geq k_0$. Thus, we find $\Delta b \geq k_0 k_1 = k > \Delta b$. This is the sought contradiction, which finishes the proof.
\end{proof}


\section{Higher degrees}\label{hd::section}

\newcommand{\cP}{\mathcal{P}}
\renewcommand{\cN}{\mathcal{N}}
\newcommand{\fpaa}[2]{\norm{#1}_{(\RR/\ZZ)^{#2}}}

In this section we deal with sets
\begin{equation}
	\A{\e}{p} := \set{ n \in \NN }{ \fpa{p(n)} \leq \e(n) },
	\label{hd::eq::def-of-A}
\end{equation}
where $p: \ZZ \to \RR$ is a polynomial, generally of degree higher than $2$, and $\e(n)$ is a slowly decaying function. Our main goal is to prove a generalisation of Theorem \ref{intro::thm:B:1}.

\begin{theorem*}[\ref{intro::thm:B:1}, reiterated]
	There exists a set $Z\subset \RR$ of measure $0$ such that for any $\e(n) > 0$ with $\frac{\log 1/\e(n)}{\log n} \to 0$ as $n \to \infty$ and any $\a \in \RR \setminus Z$, the set $\A{\e}{p}$ defined in \eqref{hd::eq::def-of-A} with $p(n) = \a n^d$ is a \bot.
\end{theorem*}

Note that the restriction $\lim \frac{\log 1/\e(n)}{\log n} \to 0$ is just another way of saying that $\e(n) = n^{-o(1)}$. In particular, any function of the form $\e(n) = \log^{-C} n$ will be suitable.

We will also give a simple argument for \ref{intro::thm:B:2}. 
\begin{theorem*}[\ref{intro::thm:B:2}, reiterated]
	There exists a closed uncountable set $E\subset \RR$ and a constant $\e_0 > 0$ such that for any $\e$ with $\e(n) \leq \e_0$ and any $\a \in E$, the set $\A{\e}{p}$ defined in \eqref{hd::eq::def-of-A} is not a \bot.
\end{theorem*}

\subsection{Bases of order 2}

In degree at least $3$, the generic behaviour is that $\A{\e}{p}$ is a \base\ of order $2$. We can prove a result for $p$ varying over an affine subspace $\cP$ of the $\RR$-vector space $\RR[x]$. For brevity, we refer to such $\cP$ as an affine family of polynomials. Note that $\cP$ has a canonical Haar measure (defined up to a constant factor), and hence we have a notion of zero measure sets.

	\begin{theorem}\label{hd::thm:B1}
	Let $\cP \subset \RR[x]$ be an affine family of polynomials, and let $\e(n) > 0$ be such that $\frac{\log 1/\e(n)}{\log n} \to 0$. Then at least of the following holds:
	\begin{enumerate}
	\item\label{hd::cond:1@thm} For all $p \in \cP$ we have $\deg p \leq 2$.
	\item\label{hd::cond:2@thm} There is $p \in \cP$ such that for all $q \in \cP$ we have $\deg p > \deg(p-q)$.
	\item\label{hd::cond:3@thm} For $p \in \cP$ except for a set of measure $0$, the set $\A{\e}{p}$ is a \bot.
	\end{enumerate}
\end{theorem}
\begin{proof}[Proof of Theorem \ref{intro::thm:B:1} assuming Theorem \ref{hd::thm:B1}]\mbox{}\\
Apply Theorem \ref{hd::thm:B1} to the linear family of polynomials $\cP = \set{\a x^d}{\a \in \RR}$, with $d \geq 3$. It is clear that neither of the conditions \eqref{hd::cond:1@thm} and \eqref{hd::cond:2@thm} holds for $\cP$. Hence, we have condition \eqref{hd::cond:3@thm}, which is precisely the claim of \ref{intro::thm:B:1}.
\end{proof}

Perhaps a more useful restatement of the above theorem is that if $\cP$ is an affine family of polynomials not satisfying \eqref{hd::cond:1@thm} and \eqref{hd::cond:2@thm}, then $\cP$ must satisfy \eqref{hd::cond:3@thm}. We clearly need to include condition \eqref{hd::cond:1@thm}, because the behaviour for polynomials of degree $2$ is different. Condition \eqref{hd::cond:2@thm} is meant to exclude the possibility that the behaviour $\A{\e}{p}$ is controlled by a highest degree term which is constant in $p$.

In the above theorem, we cannot replace ``almost all $p \in \cP$'' with ``all $p \in \cP$'', because $\A{\e}{p}$ need not be a \bot, for example, when $p$ is rational. We also believe there exist $p \in \RR[x]$ with $\deg p \geq 3$ and highly irrational leading coefficients such that $\A{\e}{p}$ is not a \bot.
	
Before we prove Theorem \ref{hd::thm:B1}, we will need a simple geometric lemma.

	\begin{lemma}\label{hd::lem:geo-equidist}
		Let $\cP$ be an affine space equipped with a volume form. Let $\a,\b\colon\cP \to \RR$ be affine forms, let $B \subset \cP$ be an open convex set, and let $k,l \in \ZZ$ be integers such that $k \a + l \b$ is non-constant. Then there exists $r_0 = r_0(\a,\b,B)$ such that for $r \geq r_0$ and arbitrary $\de > 0$ we have
	\begin{equation}
		\PP_{v \in r B} \left( \fpa{ k\a(v) + l\b(v) } \leq \delta \right) = 2 \de \left( 1 + o(1) \right)	,
	\end{equation}	
	as $r \to \infty$, where the error term is bounded uniformly in $\de$ and $k,l$ (but may depend on $\a,\b$ and $B$).
	\end{lemma}
	\begin{proof}	
	\newcommand{\vol}{\mathrm{vol}}
	If $\a$ and $\b$ are affinely dependent, then the problem becomes simpler, and can be solved by an argument similar to the one presented below.	Let us suppose that $\a,\b$ are not affinely dependent.
	
	It is easy to construct a parallelepiped $K$ such that $(\a(v),\b(v))$ are uniformly distributed in $\TT^2$ for $v \in K$. If $K$ is such  parallelepiped then 
$$ \PP_{v \in K} \left( \fpa{k \a(v) + l\b(v) } \leq \delta \right) = 2 \delta. $$
	It is elementary that for each $r$, there exist collections $C_+(r)$,  $C_-(r)$ of such parallelepipeds with $C_+(r) \supset rB \supset C_-(r)$ and $\frac{\vol(C_{\pm}(r))}{\vol(rB)} \to 1$ as $r \to \infty$. The proof now follows by a sandwiching argument.  
	\end{proof}
	
\begin{proof}[Proof of Theorem \ref{hd::thm:B1}]\mbox{}

	We shall assume that neither of the conditions \eqref{hd::cond:1@thm},  \eqref{hd::cond:2@thm} holds, and derive condition \eqref{hd::cond:3@thm}.	We may assume that $\e(n) \gg \frac{1}{n^\delta}$, where $\delta$ is a small positive constant yet to be determined, and that $\e(n)$ is decreasing.

\newcommand{\pp}{^{(p)}}		
	
	Given $p \in \cP$, we define $\cN\pp$ to be the set of $N$ such that $N \not \in 2 {} \A{\e(N)}{p}$. Since $\cN\pp \supset \NN \setminus 2 {} \A{\e}{p}$, it will suffice to show that $\cN\pp$ is almost surely finite. For this, it is enough to prove that $\EE_{p \in r B} \abs{\cN\pp} < \infty$, where $B \subset \cP$ denotes a unit ball with respect to some norm on $\cP$.

	Take any $N \in \cN\pp$. Following the argument in Lemma \ref{pos::equi::prop:nequi->ratio}, the orbit $(p(n),p(N-n))$ is not $1/\e(N)^{O(1)}$-equidistributed for $n \in [N]$. Hence, by Theorem \ref{pos::equi::thm:GT}, there exist $(k,l) \in \ZZ^2 \setminus \{(0,0)\}$ such that:
\begin{equation}
	\norm{k  p(n) + l p(N-n) }_{\cC^\infty[N]} \ll \frac{1}{\e(N)^{O(1)}} \ll N^{O(\de)}, \quad \abs{k},\abs{l} \ll N^{O(\de)},
	\label{hd::proof:cond-01}
\end{equation}
	where $\norm{ \sum_i \a_i n^i }_{\cC^\infty[N]} := \max_i N^i \fpa{\a_i} $. 
	
	Given $k,l$ and $p$, let $\cN_{k,l}\pp$ denote the set of all $N$ satisfying the above bound \eqref{hd::proof:cond-01} (for some choice of the implicit constants). We have
	$$ \cN\pp \subset \bigcup_{(k,l) \neq (0,0)} \cN\pp_{k,l}.$$
	We will allow for a certain finite set $\cN_* \subset \NN$ of $N$ which may belong to particularly many sets $\cN_{k,l}\pp$. It will suffice if (for suitable choice of $\cN_*$) we prove the bound
	\begin{equation}
	\sum_{(k,l) \neq (0,0)} \EE_{p \in r B} \abs{ \cN\pp_{k,l} \setminus \cN_* } < \infty.
	\label{hd::proof:bound-15}	
	\end{equation}

	We can write $p(x) = \sum_{i=0}^d \a\pp_i x^i$, where $\a\pp_i$ are affine functions of $p$. Because conditions  \eqref{hd::cond:2@thm} and  \eqref{hd::cond:3@thm} do not hold, $\a\pp_d$ is not a constant function of $p$, and $d \geq 3$.
	
	 A straightforward manipulation of \eqref{hd::proof:cond-01} shows that if $N \in \cN_{k,l}\pp$ then
	 \begin{align}
	 	 \fpa{ k \a\pp_d + (-1)^d l \a\pp_d} & \ll \frac{1}{N^{d-O(\de)}}, 	 	 \label{hd::proof:bound-01}\\
	 	 \fpa{k \a\pp_{d-1} - (-1)^d l \a\pp_{d-1 } - (-1)^d l \a\pp_{d} N } &  \ll \frac{1}{N^{d-1-O(\de)}}. \label{hd::proof:bound-02}	 	
	 \end{align}	 
	
	If $k + (-1)^d l \neq 0$ then the first bound \eqref{hd::proof:bound-01} together with Lemma \ref{hd::lem:geo-equidist}  implies for $r \geq r_0 = r_0(\a_d,\a_{d-1},B)$ that
	\begin{equation}
		\PP_{p \in r B}\left(N \in \cN_{k,l}\pp \right) \ll \frac{1}{N^{d - O(\de)}}.
		\label{hd::proof:bound-25}	
	\end{equation}	
	Else, if $k + (-1)^d l = 0$, then likewise the second bound \eqref{hd::proof:bound-02} together with Lemma \ref{hd::lem:geo-equidist} implies for $r \geq r_0$ that 
	\begin{equation}
		\PP_{p \in r B}\left(N \in \cN_{k,l}\pp \right) \ll \frac{1}{N^{d-1 - O(\de)}},
		\label{hd::proof:bound-03}	
	\end{equation}
	unless $2 \a_{d-1}\pp + N \a_d\pp$ is constant in $p$. The latter condition can only hold for a single value of $N$, independent of $k$ and $l$. Letting $\cN_*$ consist of this specific $N$ (or $\cN_* =\emptyset$ if no such $N$ exists), we conclude that for any $N$ we have the bound
	\begin{equation}
		\PP_{p \in r B}\left(N \in \cN_{k,l}\pp \setminus \cN_* \right) \ll \frac{1}{N^{d-1 - O(\de)}}.
		\label{hd::proof:bound-14}	
	\end{equation}
	 
	Because for $N \in \cN_{k,l}\pp$ we have $\abs{k},\abs{l} \ll N^{O(\delta)}$, at the cost of worsening implicit constants, we may rewrite \eqref{hd::proof:bound-03} as
	\begin{equation}
		\PP_{p \in r B}\left(N \in \cN_{k,l}\pp \setminus \cN_* \right) \ll \frac{1}{(k^4+l^4) N^{d-1 - O(\de)}}.
		\label{hd::proof:bound-24}	
	\end{equation}	

Taking $\delta$ sufficiently small, we can now derive
	\begin{equation}
	\sum_{k,l} \EE_{p \in r B} \abs{ \cN^p_{k,l} \setminus \cN_* } \ll \sum_{k,l} \frac{1}{k^4+l^4} \sum_{N} \frac{1}{N^{1.1}} < \infty.
	\label{hd::proof:bound-35}	
	\end{equation}	
This finishes the proof.
\end{proof}

\begin{remark*}
	The same ideas can be applied to higher dimensions. One then defines \begin{equation*}
	\A{\e}{p} := \set{ n \in \NN }{ \fpa{p_i(n)} \leq \e(n),\ i \in [r] },
\end{equation*}
for a polynomial map $p(n) = (p_i(n))_{i=1}^{r}$. For $\e(n) \geq \e_0 > 0$, these sets will generically be \bases\ of order $2$.

Because the general version of the equidistribution Theorem \ref{pos::equi::thm:GT} holds for general nilmanifolds, similar arguments can be applied to ``generic'' Nil--Bohr sets (see \cite{Green2007} and \cite{Host2009} for relevant definitions).
\end{remark*}

\subsection{Non-\bases\ of order 2}

We close this section considering situations when the sets $\A{\e}{p}$ fail to be \bases\ of order $2$. We show that for higher degrees of polynomials, it is still possible for $\A{\e}{p}$ to fail to be a \base\ of order $2$. For the sake of concreteness, we work with the polynomials of the specific form $p(n) = \a n^d$.

\begin{proof}[Proof of Theorem \ref{intro::thm:B:2}]

Take any $\e_0 < \frac{1}{4}$, and let $d$ be fixed. We first claim, in analogy to Lemma \ref{neg::prop:obst-var}, that there is some $N_0 = N_0(d,\e_0)$ such that if $N > N_0$ is odd and $\e(n) \leq \e_0$ for all $n$, and if $N $ and $\a$ satisfy
$$ \fpa{ N \a - \frac{1}{2} } \leq \frac{1}{N^d}, $$
then $N \not \in 2 {} \A{\e}{p}$. Indeed, if $n_1,n_2 \in [N]$ are such that $n_1+n_2 = N$ then
\begin{align*}
	\fpa{n_1^d \a - (-1)^d n_2^d \a} &= 
	\fpa{ \sum_{j=0}^{d-1} (-1)^j n_1^{d-1-j} n_2^j N \a  }
	= \frac{1}{2} + O\left(\frac 1N\right),
\end{align*}
where the implicit constant in the error term depends only on $d$. Thus it is impossible that $n_1,n_2 \in \A{\e}{p}$ if $\e(n) \leq \e_0 < \frac{1}{4}$ and $N$ is sufficiently large. 

Let $N_i$ be a rapidly increasing sequence of odd integers, and set
$$ \Gamma := \bigcap_{i \in \NN}\Gamma_i, \quad \Gamma_i := \set{ \a \in \TT}{ \fpa{ N_i \a - 1/2} \leq \frac{1}{N_i^d}}.$$
For any $\a \in \Gamma$, we have by the above observation that $N_i \not \in 2 {} \A{\e}{p}$ for all but finitely many $i$.

Note that for each $i$, $\Gamma_i$ is a union of $N_i$ closed intervals of length $\frac{2}{N_i^{d+1}}$, equally spaced in $\TT$. Assuming $N_i$ are increasing rapidly enough, each set $\bigcap_{j < i} \Gamma_j$ is a union of closed intervals, and each of these intervals intersects at least two different intervals in $\Gamma_i$.

It now follows easily that $\Gamma$ contains a homeomorphic copy of the Cantor set, and hence is uncountable. The set $E$ in \ref{intro::thm:B:2} can be taken to be $\Gamma + \ZZ$.
\end{proof}


\appendix
\section{Appendix: Continued fractions}\label{cf::section}

In this appendix we recall some fairly standard facts concerning continued fractions. Because the results are standard, we do not provide proofs, merely references.

\subsection{Basic definitions}

A continued fraction is an expression of the form:
$$ [a_0;a_1,a_2,\dots]
= a_0+
\cfrac{1}{a_1+\cfrac{1}{a_2+ \cfrac{1}{a_3 + \dots}  }},
$$
where $a_0 \in \ZZ$ and $a_i \in \NN$ for $i > 0$. This can be either finite or infinite; we focus mostly on the infinite case.

A standard way to make sense of infinite fractions of this form is to consider consecutive finite approximations, which we typically denote as $\frac{p_n}{q_n}$, given by:

	$$\frac{p_n}{q_n} = [a_0;a_1,a_2,\dots,a_n] = a_0 + \cfrac{1}{a_1+\cfrac{1}{ \ddots \, + \cfrac{1}{a_n} }}.$$
	In particular, $p_0 = a_0,\ q_0 = 1$. It is also convenient to define $p_{-1} = 1$ and $q_{-1} = 0$. 
	
We list some basic properties of the partial approximations. Throughout, $a_i$ denote integers, and $p_i,q_i$ are defined as above. Perhaps the most fundamental fact that we shall use is the following.

\begin{fact}[{\cite[Thm. 5]{Khinchin2009}}] \label{cf:finite}
	We have the relation $[a_0,a_1,\dots,a_{n},x] = \frac{x p_n + p_{n-1} }{x q_n + q_{n-1} }$.
\end{fact} 

It is not difficult to derive the the following consequences.

\begin{fact}[{\cite[Thm. 1, 12]{Khinchin2009}}]\label{cf:recursive}
	Sequences $p_n,q_n$ are given recursively by
	\begin{align*}
		p_{n+2} &= a_{n+2} p_{n+1} + p_n,&& p_{-1} = 1,\ p_0 = a_0, \\
		q_{n+2} &= a_{n+2} q_{n+1} + q_n,&& q_{-1} = 0,\ q_0 = 1. 		 
	\end{align*}
	In particular, $p_{n+m} \geq 2^{\frac{m-1}{2}} p_n$ and $q_{n+m} \geq 2^{\frac{m-1}{2}} q_n$ for each $n,m$.
\end{fact}

\begin{fact}[{\cite[Thm. 2]{Khinchin2009}}]\label{cf:Delta}
	We have 
	$$ q_n p_{n+1} - q_{n+1} p_n = (-1)^n,\quad \text{or equivalently:}\quad \frac{p_{n+1}}{q_{n+1}} - \frac{p_{n}}{q_n} = \frac{(-1)^n}{q_nq_{n+1}}.$$	
\end{fact}

The sequence $\frac{p_n}{q_n}$ converges rather rapidly. We shall denote
	$$\a := \lim_{n\to \infty} \frac{p_n}{q_n} = [a_0;a_1,a_2,\dots].$$
and refer to $a_i$ as the continued fraction expansion of $\a$.

\begin{fact}[{\cite[Thm. 9]{Khinchin2009}}]\label{cf:error}
	The speed of convergence above in $\a = \lim_{n\to \infty} \frac{p_n}{q_n}$ is described by:
	$$
		\a - \frac{p_n}{q_n} = \sum_{i=n}^{\infty} \frac{(-1)^i}{q_i q_{i+1}}, 
	$$
	and in particular $\abs{ \a - \frac{p_n}{q_n} } < \frac{1}{q_n q_{n+1}} < \frac{1}{a_{n+1} q_n^2}$.

\end{fact}

As a consequence, it is always easy to compare two continued fraction approximations.
\begin{fact}[{\cite[Thm. 4]{Khinchin2009}}]\label{cf:order}
 We have the ordering: 
$$\frac{p_0}{q_0} < \frac{p_2}{q_2} <  \frac{p_4}{q_4} < \dots < \a < \dots < \frac{p_5}{q_5} < \frac{p_3}{q_3} < \frac{p_1}{q_1}. $$

\end{fact}

It is often useful to have a good understanding of the ratio $\frac{q_{n+1}}{q_{n}}$. Fortunately, this quantity has a simple description.

\begin{fact}[{\cite[Thm. 6]{Khinchin2009}}]\label{cf:mirror}
	We have $\frac{q_{n+1}}{q_{n}} = [a_{n+1};a_{n},\dots,a_1]$.
\end{fact}

\subsection{Ergodic perspective}

We refer to the sequence $a_i$ as the continued fraction expansion of $\a$ above. Every irrational number has precisely one (infinite) expansion. (A similar statement is true for rational numbers, except one needs to be careful with uniqueness.) More precisely, we have the following fact.

\begin{fact}[{\cite[Lem. 3.4]{Einsiedler2010}}]
	The map $\ZZ\times \NN_{\geq 1}^\NN \colon (a_0,a_1,\dots) \mapsto [a_0;a_1,a_2,\dots,] \in \RR \setminus \QQ$ is a bijection. Likewise, the map $ \NN_{\geq 1}^\NN \colon (a_1,\dots) \mapsto [0;a_1,a_2,\dots,] \in [0,1] \setminus \QQ$ is a bijection.
\end{fact}

\begin{definition}[Continued fraction transformation]
	Define the transformation $T\colon [0,1]\setminus \QQ \to [0,1]\setminus \QQ$ by $T\alpha = \fp{\frac{1}{\alpha}}$, where $\{ x \}$ denotes the fractional part of $x$. (One may extend the definition to $\QQ$ by setting $T\alpha = 0$ for $x \in \QQ$ if one wishes to have a map $T\colon [0,1] \to [0,1]$.)
	
	Define the measure $\mu$ on $[0,1]$ by $\mu(A) = \frac{1}{\log 2} \int_{A} \frac{dx}{x+1}$ for $A \in \Borel([0,1]) $, where $\Borel([0,1])$ denotes the Borel $\sigma$-algebra.
	
	We refer to the transformation $T$ as the \emph{continued fraction transformation} and to  $\mu$ as the \emph{Gauss measure}.
\end{definition}

\begin{fact}[{\cite[Chpt. 3]{Einsiedler2010}}]
	The transformation $T$ acts on $\NN_{\geq 1}^{\NN}$ by a shift:
	$$
		T([0;a_1,a_2,\dots]) = T([0;a_2,a_3,\dots]).
	$$
\end{fact}

\begin{fact}[{\cite[Chpt. 3]{Einsiedler2010}}]
	The measure $\mu$ is equivalent to the Lebesgue measure. The transformation $T$ is measurable and piecewise continuous.
\end{fact}

\begin{fact}[{\cite[Chpt. 3]{Einsiedler2010}}]
	The transformation $T$ is $\mu$-invariant, in the sense that for each $A \in \Borel([0,1])$ we have $\mu(T^{-1}(A)) = \mu(A)$. 
\end{fact}

Thus, $([0,1],T,\Borel([0,1]),\mu)$ is a measure preserving system (for introduction to measure preserving systems, see e.g. \cite[Chapter 1]{Einsiedler2010}).

\begin{fact}[{\cite[Thm. 3.7]{Einsiedler2010}}]
	The measure preserving system  $([0,1], T, \Borel([0,1]),  \mu)$ is ergodic.
\end{fact}

\subsection{Good rational approximations}

Essentially all good rational approximations of a number $\a = [a_0;a_1,a_2,\dots]$ come from continued fractions.

\begin{fact}[Legendre, {\cite[Thm. 4]{Khinchin2009}}]\label{cf:Legendre}
	If $\abs{ \alpha - \frac{p}{q} } < \frac{1}{2q^2}$ for some $\frac{p}{q}$, then there exists some $i$ with  $\frac{p}{q} = \frac{p_i}{q_i}$.  
\end{fact}

For context, we also mention a result which we don't use, even implicitly.
\begin{fact}[Hurwitz]\label{cf:Hurwitz}
	For every $\alpha$ there are $p,q$ such that $\abs{ \alpha - \frac{p}{q} } < \frac{1}{\sqrt{5}q^2}$.  
\end{fact}

Call a fraction $\frac{p}{q}$ a \emph{best rational approximation} (of the second kind, in terminology of \cite{Khinchin2009}) of $\a$ if $\abs{ q \a - p } < \abs{q' \a - p'}$ for any $\frac{p'}{q'} \neq \frac{p}{q} $ with $q' \leq q$.

\begin{fact}[{\cite[Thm. 16, 17]{Khinchin2009}}]\label{cf:best-approx}
	If $\frac{p}{q}$ is a best rational approximation of $\a$, then there exists some $i$ with  $\frac{p}{q} = \frac{p_i}{q_i}$. Conversely, if $i \geq 1$ then $\frac{p_i}{q_i}$ is a best rational approximation of $\a$. In particular, if $\abs{\a - \frac{p}{q}} \leq \abs{\a - \frac{p_i}{q_i}}$ then $q \geq q_i$.
\end{fact}

See also Chapter 6 of {\cite{Khinchin2009}} for different notions of a best rational approximation and more similar results.

Badly approximable numbers can be characterised in terms of their continued fraction expansion. Recall that $\a$ is badly approximable precisely when $\abs{\a - \frac{p}{q}} \gg \frac{1}{q^2}$ for all $\frac{p}{q}$.

\begin{fact}[{\cite[Prop. 3.10]{Einsiedler2010}}]\label{cf:bad-approx}
	The number $\a$ is badly approximable if and only if the sequence $(a_i)_{i=1}^\infty$ is bounded.
\end{fact}

A particularly important class of badly approximable are the quadratic irrationals.

\begin{fact}[{\cite[Thm. 28]{Khinchin2009}}]\label{cf::quadratic-irrational}
	The expansion $(a_i)_{i=1}^\infty$ of $\alpha$ is eventually periodic if and only if $\alpha$ is a quadratic irrational.
	
	In particular, if $\alpha$ is quadratic irrational then $\a$ is badly approximable.
\end{fact}

\bibliographystyle{abbrv}
\bibliography{bibliography}

\end{document}